\documentclass[11pt]{article}
\usepackage[top=1.5cm, bottom=1.5cm, left=1.5cm, right=1.5cm]{geometry}
\usepackage[titletoc,title]{appendix}
\usepackage{graphicx}
\usepackage{color}
\usepackage{fancyhdr,hyperref}
\usepackage{amsmath,amssymb,amsthm}
\usepackage[font=small,skip=0pt]{caption}
\newtheorem{defn}{Definition}
\newtheorem{thm}{Theorem}
\newtheorem{lem}{Lemma}
\newtheorem{cor}{Corollary}

\DeclareMathOperator*{\adjugate}{adj}

\newcommand\norm[1]{\left\lVert#1\right\rVert}
\newcommand\cont[1]{[#1]}
\newcommand\disc[1]{\{#1\}}
\newcommand\contdep[2]{[#1]_{#2}}
\newcommand\discdep[2]{\{#1\}_{#2}}
\newcommand\mesh{\mathcal{I}^\tau\times\Omega^h}

\newcommand\aw[1]{{#1}}

\newcommand\rev[1]{{#1}}

\title{The multiplier method to construct conservative \\finite difference schemes for ordinary and partial differential equations}
\author{Andy T.S. Wan\thanks{Department of Mathematics and Statistics, McGill University, Montr\'eal,  QC, H3A 0B9, Canada (\href{mailto:andy.wan@mcgill.ca}{andy.wan@mcgill.ca})}, Alexander Bihlo\thanks{Department of Mathematics and Statistics, Memorial University of Newfoundland, St. John's, NL, A1C 5S7, Canada (\href{mailto:abihlo@mun.ca}{abihlo@mun.ca})}, Jean-Christophe Nave\thanks{Department of Mathematics and Statistics, McGill University, Montr\'eal,  QC, H3A 0B9, Canada (\href{mailto:jcnave@math.mcgill.ca}{jcnave@math.mcgill.ca})}}

\begin{document}
\maketitle

\begin{abstract}
We present the multiplier method of constructing conservative finite difference schemes for ordinary and partial differential equations. Given a system of differential equations possessing conservation laws, our approach is based on discretizing conservation law multipliers and their associated density and flux functions. We show that the proposed discretization is consistent for any order of accuracy when the discrete multiplier has a multiplicative inverse. Moreover, we show that by construction, discrete densities can be exactly conserved. In particular, the multiplier method does not require the system to possess a Hamiltonian or variational structure. Examples, including dissipative problems, are given to illustrate the method. In the case when the inverse of the discrete multiplier becomes singular, consistency of the method is also established for scalar ODEs provided the discrete multiplier and density are zero-compatible.\end{abstract}


\section{Introduction}
In recent decades, structure preserving discretizations have become an important idea in designing numerical methods for differential equations. The central theme is to devise discretizations which preserve at the discrete level certain important structures belonging to the continuous problem.

In the case of ordinary differential equations with a Hamiltonian structure, geometric integrators such as, energy-momentum method \cite{LaBGre75,SimTarWon98}, symplectic integrators \cite{leim04, hair06},  variational integrators \cite{MarWes01}, Lie group methods \cite{Dor10, Hyd14} and method of invariantization \cite{Olv01, KimOlv04} are a class of discretizations which preserves symplectic structure, first integrals, phase space volume or symmetries at the discrete level. Multi-symplectic methods have also been extended to include partial differential equations possessing a Hamiltonian structure \cite{Bri97, BriRei06, MarPatShk98}.

On the other hand for partial differential equations in divergence form, the well-known finite volume methods \cite{Lev92} naturally preserves conserved quantities in discrete subdomains for the equations. More recently, structure preserving discretizations such as discrete exterior calculus \cite{Hir03,DesAE05}, mimetic discretizations \cite{boch06, RobSte11} and finite element exterior calculus \cite{ArnFalWin06,ArnFalWin10} put forth complete abstract theories on discretizing equations with a differential complex structure. Specifically, these approaches provide consistency, stability and preservation of the key structures of de Rham cohomology and Hodge theory at the discrete level.

In the present work, we propose a numerical method for ordinary differential equations and partial differential equations which preserves discretely their first integrals, or more generally, their conservation laws. In contrast to geometric integrators or multi-symplectic methods which require the equations to possess a Hamiltonian or variational structure, our approach is based solely on the existence of conservation laws for the given set of differential equations and thus is applicable to virtually any physical system in practice. In particular, this method can be applied to dissipative problems possessing conserved quantities.

For a given system of differential equations with a set of conservation laws, under some mild assumptions on the local solvability of the system, there exists so-called conservation law multipliers which are associated to each conservation law and vice versa \cite{Olv00, BluAnc10}. This can be viewed as a generalization to the well-known Noether's theorem for equations derived via a variational principle, in which conservation laws of the Euler-Lagrange equations are associated with variational symmetries and vice versa.

\rev{
The main idea of our approach, called the multiplier method, is to directly discretize the conservation laws and their associated conservation law multipliers in a manner which is consistent to the given set of differential equations. Moreover, this approach provides a general framework for which discrete densities can be conserved exactly by construction. We emphasize the novelty of the multiplier method is that one can discretize a system of differential equations, which may not be in a conserved form, so that its conserved quantities are simultaneously conserved at the discrete level. 
}

The content of this paper is as follows: 
First, in Section \ref{sec:multiplierIntro}, we provide a brief introduction to conservation law multipliers and the relevant background. Then, we describe our method in Section \ref{sec:scalar} for the case of a scalar partial differential equation, as most of the main ideas are already present in that case. In Section \ref{sec:squareSystem}, we generalize the method to the case when there is the same number of equations as the number of conservation laws. In Section \ref{sec:rectSystem}, we treat the case when there are more equations than the number of conservation laws. In Section \ref{sec:examples}, we illustrate the multiplier method with various examples and demonstrate discrete preservation of their associated conservation laws; this includes problems which do not readily possess a Hamiltonian or variational structure. \rev{Finally in Section \ref{sec:singularMult}, we define the zero-compatible condition for scalar ODEs in the case when the inverse of multiplicative discrete multipliers become singular and establish the consistency of the multiplier method in that case.}

\section{Conservation law of PDEs and conservation law multipliers}
\label{sec:multiplierIntro}
Before we describe the construction of conservative schemes for PDEs, we introduce some background for conservation laws of PDEs. First, some notations and definitions. \rev{Let $\mathcal{I}=(0,T)$, $\mathcal{V}$ be an open neighborhood of a bounded domain $\Omega$ on $\mathbb{R}^n$ and $\mathcal{U}^{(i)}$ be open neighborhoods of $\mathbb{R}^{m(n+1)^{i}}$ for $i\in \mathbb{N}$. We express a $k$-th order system of PDEs as a function $\boldsymbol F:I\times \mathcal{V} \times \mathcal{U}^{(0)} \times \cdots \times \mathcal{U}^{(k)} \rightarrow  \mathbb{R}^m$,
\begin{equation*}{\boldsymbol F}\contdep{\boldsymbol u}{t,\boldsymbol x}:={\boldsymbol F}(t,{\boldsymbol x};{\boldsymbol u}(t,{\boldsymbol x}),D^1{\boldsymbol u}(t,{\boldsymbol x}),\dots,D^k{\boldsymbol u}(t,{\boldsymbol x}))=0,\end{equation*}where $t\in I$, ${\boldsymbol x}=(x_1,\dots,x_n)\in \mathcal{V}$ are independent variables and ${\boldsymbol u}:I\times \mathcal{V} \rightarrow \mathbb{R}^m$ are the dependent variables with components ${\boldsymbol u}=(u^1,\dots,u^m)\in \mathcal{U}^{(0)}$ and $D^{i}\boldsymbol u \in \mathcal{U}^{(i)}$ are all partial derivatives of $\boldsymbol u$ up to order $i$ with respect to $t, \boldsymbol x$, for $i=1,\dots,k$. To save writing, we introduced the square bracket notation ${\boldsymbol F}\contdep{\boldsymbol u}{t,\boldsymbol x}$ to denote the value of $\boldsymbol F$, depending on $t, \boldsymbol x,$ and $\boldsymbol u, D^{1}\boldsymbol u, \dots, D^{k}\boldsymbol u$, evaluated at $t, \boldsymbol x$.} If $\boldsymbol F$ is an analytic normal system\footnote{The analyticity condition can be relaxed; see normal, nondegenerate system of PDEs in \cite{Olv00}.} \cite{Olv00} \aw{for which the Cauchy-Kovalevskaya existence theorem is applicable}, then there exist local analytic solutions $\boldsymbol u$ on $\mathcal{I}\times \mathcal{V}$ which allow us to differentiate $\boldsymbol u$ as many times as we wish on $\mathcal{I}\times \mathcal{V}$.
\begin{defn}A $l$-th order conservation law of $\boldsymbol F$ is a divergence expression which vanishes on the solutions of $\boldsymbol F$,
\begin{equation*}\rev{\left.\left(\frac{}{}D_t \psi\contdep{\boldsymbol u}{t,\boldsymbol x}+ D_{\boldsymbol x}\cdot \boldsymbol\phi\contdep{\boldsymbol u}{t,\boldsymbol x}
\right)\frac{}{}\right|_{{\boldsymbol F\contdep{\boldsymbol u}{t,\boldsymbol x}}=0}=0,}\end{equation*} for some density function $\psi$ and flux functions $\boldsymbol \phi=(\phi^1, \dots, \phi^n)$ which are analytic on $\mathcal{I}\times \mathcal{V}\times \mathcal{U}^{(0)} \times \cdots \times  \mathcal{U}^{(l)}$.\end{defn}\aw{\noindent Here, $D_t$ is the total derivative with respect to $t$ and the notation  $D_{\boldsymbol x} \cdot \boldsymbol \phi$ means $D_{\boldsymbol x} \cdot \boldsymbol \phi = \sum_{i=1}^n D_{x_i}\phi^{i}$, where $D_{x_i}$ is the total derivative with respect to $x_i$.}

Clearly, adding two conservation laws together yields a conservation law and multiplying by a scalar constant also yields a conservation law. So, the set of conservation laws of $\boldsymbol F$ forms a vector space. However, some conservation laws are not as meaningful as others, for example: density and fluxes which together form differential identities, such as $D_t( u_x) + D_x(u_t)=0$, or density and fluxes which are differential consequences of the PDEs $\boldsymbol F$ itself, such as if $\psi$ and $\phi^i$ are of the form $\sum_{j,\boldsymbol \alpha} A_{j,\boldsymbol \alpha} D_{\boldsymbol \alpha} F^j$ for some analytic $A_{j,\boldsymbol \alpha}$. These two types of conservation laws are classified as \emph{trivial} conservation laws of $\boldsymbol F$. Since we are interested in non-trivial conservation laws, we ``mod out" the trivial conservation laws by considering two conservation laws as equivalent if the difference in their density and fluxes yields a trivial conservation law. It can be shown such relation defines an equivalence relation on the set of conservation laws of $\boldsymbol F$.

\begin{defn}An $m$-tuple of analytic functions $\boldsymbol \lambda:\mathcal{I}\times \mathcal{V} \times \mathcal{U}^{(0)} \times \cdots \times \mathcal{U}^{(r)}\rightarrow \mathbb{R}^m$ is called a $r$-th order conservation law multiplier of $\boldsymbol F$ if there exists a density function $\psi$ and flux functions $\boldsymbol \phi$ which are analytic on $\mathcal{I}\times \mathcal{V}\times \mathcal{U}^{(0)} \times \cdots \times \mathcal{U}^{(l)}$ such that,
\begin{equation*} \rev{\boldsymbol \lambda\contdep{\boldsymbol u}{t,\boldsymbol x} \cdot \boldsymbol F\contdep{\boldsymbol u}{t,\boldsymbol x} = D_t \psi\contdep{\boldsymbol u}{t,\boldsymbol x}+ D_{\boldsymbol x} \cdot \boldsymbol \phi\contdep{\boldsymbol u}{t,\boldsymbol x},}\end{equation*} holds on $\mathcal{I}\times \mathcal{V}\times \mathcal{U}^{(0)} \times \cdots \times \mathcal{U}^{(l)}$, where $l=max(r,k)-1$ and the $``\cdot"$ symbol denotes the usual dot product. \end{defn}

It follows that the existence of a conservation law multiplier of $\boldsymbol F$ implies the existence of a conservation law of $\boldsymbol F$. The converse is also true if the PDE system $\boldsymbol F$ is an analytic normal system.

\begin{thm}[\cite{Olv00}]
Let $\boldsymbol F$ be an analytic normal system and suppose a conservation law of $\boldsymbol F$ is given by the densities $\psi$ and fluxes $\boldsymbol \phi$. Then there exists a conservation law multiplier $\boldsymbol \lambda$ such that,
\begin{equation*} \rev{\boldsymbol \lambda\contdep{\boldsymbol u}{t,\boldsymbol x} \cdot \boldsymbol F\contdep{\boldsymbol u}{t,\boldsymbol x} = D_t \widetilde{\psi}\contdep{\boldsymbol u}{t,\boldsymbol x}+ D_{\boldsymbol x} \cdot \widetilde{\boldsymbol \phi}\contdep{\boldsymbol u}{t,\boldsymbol x},}\end{equation*} holds on $\mathcal{I}\times \mathcal{V}\times \mathcal{U}^{(0)} \times \cdots \times \mathcal{U}^{(l)}$, where $\widetilde{\psi}$ and $\widetilde{\boldsymbol \phi}$ are equivalent density and fluxes to $\psi$ and $\boldsymbol \phi$, respectively. \label{equiThm}
\end{thm}

Thus, Theorem \ref{equiThm} provides a recipe to find conservation laws of PDEs. In other words, given a fixed order $r$, one can in principle find all conservation law of order $r$ by first finding the set of conservation law multipliers (which might be empty for a given $r$), and then compute their associated densities and fluxes. Conservation law multipliers can be computed by using the method of Euler operator \cite{Olv00,BluAnc10}. Several techniques on computing the density and fluxes associated with a given conservation law multiplier have also been discussed in \cite{BluAnc10}.

\section{Conservative finite difference schemes for PDEs}
\label{sec:consFD}
We now describe the theory on how to use conservation law multipliers and their associated density and fluxes to construct conservative finite difference schemes. \aw{For clarity of presentation, we have restricted to the case of a uniform spatial mesh and uniform time step, though in principle, extension to non-uniform grids and curved domains is possible. 

The main idea is the following: given a system of partial differential equations with a conservation law, we propose a consistent discretization for the equations by approximating its associated conservation law multipliers, density and fluxes. Moreover, we show the proposed discretization satisfies a discrete divergence theorem. In particular, this implies exact conservation of discrete densities in a manner analogous to the telescoping sum for the finite volume method. 

There are three parts in this section. First, we will present the main results in the scalar case. Second, we then generalize to the vectorial case when the number of conservation laws is the same as the number of equations. Finally, we further generalize the vectorial case when the number of conservation laws is less than the number of equations. The case when the number of conservation laws is greater than the number of equations requires a more delicate treatment; we will elaborate in the conclusion.\\

\rev{
\noindent \textbf{Notation.} In this section, we let $\mathcal{I}^\tau=\{k\tau: k=0,\dots, N:=\lfloor T/\tau \rfloor\}$ be uniform time steps and denote $\tau$ as the time step size. We denote a uniform mesh $\Omega^h:=\Omega \cap \mathbb{R}^{n,h}$, where $\mathbb R^{n,h}=\{h \vec{z}\in \mathbb{R}^n: \vec{z}\in \mathbb{Z}^n\}$ are the lattice points in $\mathbb{R}^n$ and $h$ is the spatial mesh size. To avoid technicality arising from boundary points of $\Omega$ not being on $\Omega^h$, we shall assume $h$ to be sufficiently small such that the discretization error of $\partial \Omega$ with $\partial \Omega^h$ to be negligible.
Denote the mesh points of $\Omega^h$ by ${\boldsymbol x}_J=(x_{j_1},\dots,x_{j_n})$, where $J=(j_1,\dots,j_n)\in \mathbb{Z}^n$ is a mesh multi-index if ${\boldsymbol x}_J \in \Omega^h$. Specifically, given a mesh multi-index $J$, by fixing $J_i=(j_1,\dots j_{i-1}, j_{i+1}, \dots, j_n)$, then $j_i$ can range over a union of $M_{J_i}$ consecutive indices $\{\alpha_{J_i,l}, \alpha_{J_i,l}+1, \dots, \beta_{J_i,l}-1 , \beta_{J_i,l}\}$, for $l=1,\dots, M_{J_i}$; i.e. $j_i \in \bigcup_{l=1}^{M_{J_i}} \{\alpha_{J_i,l},\dots,\beta_{J_i,l}\}$. In particular, if $\Omega$ is convex, then $M_{J_i}=1$ and $\alpha_{J_i}\leq j_i \leq \beta_{J_i}$. Thus, the boundary of $\Omega^h$ can then be defined as, 
\begin{equation}
\partial\Omega^h=\bigcup_{i=1}^n \{\boldsymbol x_J\in \Omega^h : J=(j_1,\dots,j_n) \text{ such that for some }l \in \{1,\dots,M_{J_i}\},  j_i=\alpha_{J_i,l}\text{ or }\beta_{J_i,l}\}. \label{boundary}
\end{equation}  We use $\hat{i}$ to denote the multi-index $(0,\dots,1,\dots,0)$ with $1$ at the $i$-th component.

Since finite difference discretizations are expressions which depend on values defined on mesh points, we denote the collection of mesh values of a discrete function $\boldsymbol u^{\tau, h}: \mesh \rightarrow \mathbb{R}^m$ as,
$$
\{\boldsymbol u^{\tau,h}\} := (\dots, \boldsymbol u_{k+1,J}, \dots, \boldsymbol u_{k,J+\hat{i}}, \boldsymbol u_{k,J}, \boldsymbol u_{k,J-\hat{i}}, \dots, \boldsymbol u_{k-1,J}, \dots),
$$ and denote a finite difference discretization $\boldsymbol F^{\tau,h}$ depending on $\{\boldsymbol u^{\tau,h}\}$ at $t_k, \boldsymbol x_J$ as,
$$
\boldsymbol F^{\tau,h}\discdep{\boldsymbol u^{\tau,h} }{k,J} := \boldsymbol F^{\tau,h}(t_k, \boldsymbol x_J;\dots, \boldsymbol u_{k+1,J}, \dots, \boldsymbol u_{k,J+\hat{i}}, \boldsymbol u_{k,J}, \boldsymbol u_{k,J-\hat{i}}, \dots, \boldsymbol u_{k-1,J},\dots).
$$
In order to discuss the consistency of finite difference discretization for $\boldsymbol F^{\tau,h}$, we also denote the expression of $\boldsymbol F^{\tau,h}$ when evaluated on the mesh values of any continuous function $\boldsymbol u\in C(\overline{\mathcal{I}\times\Omega})$ as,
$$
\boldsymbol F^{\tau,h}\discdep{\boldsymbol u}{t_k,\boldsymbol x_J} := \boldsymbol F^{\tau,h}(t_k, \boldsymbol x_J;\dots, \boldsymbol u(t_{k+1},\boldsymbol x_{J}), \dots, \boldsymbol u(t_{k},\boldsymbol x_{J+\hat{i}}), \boldsymbol u(t_k,\boldsymbol x_{J}), \boldsymbol u(t_{k},\boldsymbol x_{J-\hat{i}}), \dots, \boldsymbol u(t_{k-1},\boldsymbol x_{J}), \dots).
$$
Similarly, we shall use the same notations for finite difference discretizations of multipliers $\boldsymbol \lambda^{\tau, h}$, total derivative of the density $D_t^{\tau, h} \psi$ and total derivative of the fluxes $D_{\boldsymbol x}^{\tau, h}\boldsymbol \phi$.
Often to save writing, we will omit the subscripts $k,J$ or $t_k, \boldsymbol x_J$ when the specific mesh point is clear from the context.
}
}
\subsection{Scalar PDE with one conservation law}
\label{sec:scalar}
\rev{
We first consider the scalar case when $\boldsymbol u$ has only one component $u$,
\begin{equation}F\contdep{u}{t,\boldsymbol x}=0, \text{ in } t\in \mathcal{I}, {\boldsymbol x} \in \Omega, \label{scalarPDE} \end{equation} with one conservation law. In particular, we have a conservation law multiplier $\lambda$ with their associated densities $\psi$ and fluxes ${\boldsymbol\phi}$ satisfying,
\begin{equation*} \lambda\contdep{u}{t,\boldsymbol x} \cdot F\contdep{u}{t,\boldsymbol x} = D_t \psi\contdep{u}{t,\boldsymbol x}+ D_{\boldsymbol x}\cdot {\boldsymbol\phi}\contdep{u}{t,\boldsymbol x}.\end{equation*}

\begin{thm} Let $u \in C^{\max(p,q)+\max(r,k)}(\overline{\mathcal{I}\times \Omega})$\footnote{\rev{Recall, we assumed $\lambda$ has up to $r$-th order derivatives and $F$ has up to $k$-th order derivatives. Moreover, additional $\max(p,q)$ derivatives are necessary to state consistency for $\lambda^{\tau,h}$, $D^{\tau,h}_t\psi$ and $D_{\boldsymbol x}^{\tau, h}\cdot \boldsymbol\phi$ required for the use of Taylor's Theorem.}} and $\lambda$ be a conservation law multiplier of \eqref{scalarPDE} with corresponding density $\psi$ and fluxes $\boldsymbol \phi$. Suppose $\lambda^{\tau, h}$, $D_t^{\tau, h}\psi$ and $D_{\boldsymbol x}^{\tau, h}\cdot \boldsymbol\phi$ are finite difference discretizations of $\lambda$, $D_t\psi$ and $D_{\boldsymbol x}\cdot \boldsymbol\phi$ with accuracy up to order $p$ in space and $q$ in time:
\begin{align*} 
\norm{\lambda\cont{u}-\lambda^{\tau, h}\disc{u}}_{l^\infty(\mesh)} & \leq C_\lambda(h^p+\tau^q), \\
\norm{D_t\psi\cont{u}-D_t^{\tau, h}\psi\disc{u}}_{l^\infty(\mesh)}  & \leq C_t(h^p+\tau^q), \\
\norm{D_{\boldsymbol x}\cdot \boldsymbol\phi\cont{u}-D_{\boldsymbol x}^{\tau, h}\cdot \boldsymbol\phi\disc{u}}_{l^\infty(\mesh)}  & \leq C_x(h^p+\tau^q).
\end{align*} 
Let $u^{\tau,h}:\mesh\rightarrow \mathbb{R}$ be a discrete scalar function. Suppose ${(\lambda^{\tau, h})}^{-1}\disc{u^{\tau,h}}$, the multiplicative inverse of $\lambda^{\tau, h}\disc{u^{\tau,h}}$, exists on $\mesh$ and that $\norm{{(\lambda^{\tau, h})}^{-1}\disc{u^{\tau,h}}}_{l^\infty(\mesh)}\leq \gamma<\infty$, with $\gamma$ independent of $\tau, h$.
Define the finite difference discretization of $F$ as,
\begin{equation} F^{\tau, h}\disc{u^{\tau,h}}:={(\lambda^{\tau, h})}^{-1}\disc{u^{\tau,h}}\left(D_t^{\tau, h}\psi\disc{u^{\tau,h}}+D_{\boldsymbol x}^{\tau, h}\cdot \boldsymbol\phi\disc{u^{\tau,h}}\right).\label{discScalarF}\end{equation} Then there exists a constant $C>0$ independent of $h,\tau$ such that,
\begin{equation*} \norm{F\cont{u}-F^{\tau, h}\disc{u}}_{l^\infty(\mesh)} \leq C(h^p+\tau^q)\end{equation*}\label{scalarThm}
\end{thm}\begin{proof}Fix $(t_k,\boldsymbol x_J)\in \mesh$. Since $\lambda$ is a conservation law multiplier of $F$ and by definition of $F^{\tau, h}$,
\begin{align}
F\cont{u}&-F^{\tau,h}\disc{u} \nonumber \\
&= {(\lambda^{\tau, h})}^{-1}\disc{u}\left( \lambda^{\tau, h}\disc{u}F\cont{u}-D_t^{\tau, h}\psi\disc{u}-D_{\boldsymbol x}^{\tau, h}\cdot \boldsymbol\phi\disc{u}\right) \nonumber \\
&={(\lambda^{\tau, h})}^{-1}\disc{u}\left( (\lambda^{\tau, h}\disc{u}-\lambda\cont{u})F\cont{u}+\left(D_t \psi\cont{u}-D_t^{\tau, h}\psi\disc{u}\right)+\left(D_{\boldsymbol x}\cdot \boldsymbol\phi\cont{u}-D_{\boldsymbol x}^{\tau, h}\cdot \boldsymbol\phi\disc{u}\right)\right) \label{t1e1}
\end{align}
Since $u\in C^{\max(p,q)+\max(r,k)}(\overline{\mathcal{I}\times \Omega})$ and $F[u]$ is continuous in its arguments, $F\contdep{u}{t,\boldsymbol x}$ is continuous on $\overline{\mathcal{I}\times \Omega}$ and thus $|F\contdep{u}{t,\boldsymbol x}|\leq C'$ uniformly on $\overline{\mathcal{I}\times \Omega}$. By \eqref{t1e1},
\begin{align}
\left| F\cont{u}-F^{\tau,h}\disc{u}\right| &\leq |{(\lambda^{\tau, h})}^{-1}\disc{u}|\left( \left|\lambda^{\tau, h}\disc{u}-\lambda\cont{u}\right| |F\cont{u}|+\left|D_t \psi\cont{u}-D_t^{\tau, h}\psi\disc{u}\right|+\left|D_{\boldsymbol x}\cdot \boldsymbol\phi\cont{u}-D_{\boldsymbol x}^{\tau, h}\cdot \boldsymbol\phi\disc{u}\right|\right) \nonumber \\
&\leq \gamma(C'C_{\lambda}(h^p+\tau^q))+C_{t}(h^p+\tau^q)+C_x(h^p+\tau^q)) = C(h^p+\tau^q), \label{t1e2}
\end{align} where $\displaystyle C:=\gamma\left(C'C_{\lambda}+C_{t}+C_{x}\right)$. Since this is true for arbitrary mesh points $(t_k,\boldsymbol x_J)\in \mesh$, taking the maximum of \eqref{t1e2} over all points in $\mesh$ gives the desired result.

\end{proof} In other words, Theorem \ref{scalarThm} says the discretization given by \eqref{discScalarF} is a consistent discretization of \eqref{scalarPDE}, provided $\lambda^{\tau, h}\disc{u^{\tau,h}}$ does not vanish on some $(t_k,\boldsymbol x_J)\in \mesh$. We will return in Section \ref{sec:singularMult} to address the issue of when $\lambda^{\tau, h}\disc{u^{\tau,h}}$ do vanish. Next we show that the finite difference discretization $F^{\tau, h}$ is conservative in the following sense.
\begin{thm}
Let $\lambda$ be a conservation law multiplier of \eqref{scalarPDE} with density $\psi$ and fluxes $\boldsymbol \phi$ and let $\lambda^{\tau, h}$ be a finite difference discretizations of $\lambda$. Let $u^{\tau,h}:\mesh\rightarrow \mathbb{R}$ be a discrete scalar function. Define $D_t^{\tau, h}\psi$ and $D_{\boldsymbol x}^{\tau, h}\cdot \boldsymbol\phi$ to be the following discretizations,
\begin{subequations}
\begin{align}
D_t^{\tau, h}\psi\discdep{u^{\tau,h}}{k,J} &= \frac{\psi^{\tau, h}\discdep{u^{\tau,h}}{k,J}-\psi^{\tau, h}\discdep{u^{\tau,h}}{k-1,J}}{\tau}, \label{dtScalar} \\
D_{\boldsymbol x}^{\tau, h}\cdot \boldsymbol\phi\discdep{u^{\tau,h}}{k,J} &= \sum_{i=1}^n\frac{\phi^{i,\tau, h}\discdep{u^{\tau,h}}{k,J}-\phi^{i,\tau, h}\discdep{u^{\tau,h}}{k,J-\hat{i}}}{h}, \label{dxScalar}
\end{align}
\end{subequations}where $\psi^{\tau, h}$ and $\phi^{i,\tau, h}$ are finite difference discretizations of $\psi$ and $\phi^i$. Also let $F^{\tau, h}$ be the corresponding discretization of \eqref{scalarPDE} given by \eqref{discScalarF}. If $u^{\tau,h}$ is a solution to the discrete problem, \begin{equation}F^{\tau, h}\discdep{u^{\tau,h}}{k,J}=0,\hspace{3mm} {\boldsymbol x}_J\in \Omega^h, \label{discScalar}\end{equation} then the discrete divergence theorem for $u^{\tau,h}$ holds,
\begin{align*}
0=\frac{1}{\tau}\sum_{{\boldsymbol x}_J \in \Omega^h} \left(\psi^{\tau, h}\discdep{u^{\tau,h}}{k,J}-\psi^{\tau, h}\discdep{u^{\tau,h}}{k-1,J}\right)+\frac{1}{h} \sum_{{\boldsymbol x}_J \in \partial\Omega^h} \boldsymbol \phi^{\tau, h}\discdep{u^{\tau,h}}{k,J} \cdot \boldsymbol \nu^J, 
\end{align*}where $\boldsymbol \nu^J=(\nu_1^J,\dots,\nu_n^J)$ is the outward-pointing unit normal of $\partial\Omega^h$ at $\boldsymbol x_J$. \label{scalarCL}
\end{thm}
\begin{proof} Since $u^{\tau,h}$ satisfies \eqref{discScalar} for all $\boldsymbol x_J\in \Omega^h$ and by definition of $F^{\tau, h}$,
\begin{align}
0&=\sum_{{\boldsymbol x}_J \in \Omega^h} \lambda^{\tau, h}\discdep{u^{\tau,h}}{k,J} F^{\tau, h}\discdep{u^{\tau,h}}{k,J} = \sum_{{\boldsymbol x}_J \in \Omega^h} \left(D_t^{\tau, h}\psi\discdep{u^{\tau,h}}{k,J}\right)+\sum_{{\boldsymbol x}_J \in \Omega^h}\left(D_{\boldsymbol x}^{\tau, h}\cdot \boldsymbol\phi\discdep{u^{\tau,h}}{k,J}\right) \nonumber \\
&=\frac{1}{\tau}\sum_{{\boldsymbol x}_J \in \Omega^h} \left(\psi^{\tau, h}\discdep{u^{\tau,h}}{k,J}-\psi^{\tau, h}\discdep{u^{\tau,h}}{k-1,J}\right)+\frac{1}{h}\sum_{i=1}^n\sum_{{\boldsymbol x}_J \in \Omega^h} \left(\phi^{i,\tau, h}\discdep{u^{\tau,h}}{k,J}-\phi^{i,\tau, h}\discdep{u^{\tau,h}}{k,J-\hat{i}}\right).  \label{t2e1}
\end{align} For each fixed $i$ and $J_i=(j_1,\dots,j_{i-1},j_{i+1},\dots j_n)$, recall that $j_i \in \bigcup_{l=1}^{M_{J_i}} \{\alpha_{J_i,l},\dots,\beta_{J_i,l}\}$. Thus,
\begin{align*}
&\sum_{{\boldsymbol x}_J \in \Omega^h} \left(\phi^{i,\tau, h}\discdep{u^{\tau,h}}{k,J}-\phi^{i,\tau, h}\discdep{u^{\tau,h}}{k,J-\hat{i}}\right) = \sum_{\substack{j_1,\dots,j_n}} \left(\phi^{i,\tau, h}\discdep{u^{\tau,h}}{k,(j_1,\dots,j_i,\dots, j_n)}-\phi^{i,\tau, h}\discdep{u^{\tau,h}}{k,(j_1,\dots,j_i-1,\dots, j_n)}\right) \\
&= \sum_{j_1,\dots,j_{i-1},j_{i+1},\dots, j_n} \sum_{l=1}^{M_{J_i}}\sum_{j_i=\alpha_{J_i,l}+1}^{\beta_{J_i,l}} \left(\phi^{i,\tau, h}\discdep{u^{\tau,h}}{k,(j_1,\dots,j_i,\dots, j_n)}-\phi^{i,\tau, h}\discdep{u^{\tau,h}}{k,(j_1,\dots,j_i-1,\dots, j_n)}\right)  \\
&= \sum_{j_1,\dots,j_{i-1},j_{i+1},\dots, j_n} \sum_{l=1}^{M_{J_i}}\left(\phi^{i,\tau, h}\discdep{u^{\tau,h}}{k,(j_1,\dots,\beta_{J_i,l},\dots, j_n)}-\phi^{i,\tau, h}\discdep{u^{\tau,h}}{k,(j_1,\dots,\alpha_{J_i,l},\dots, j_n)}\right)\\
&= \sum_{\substack{{\boldsymbol x}_J \in \Omega^h\\ j_i\in \bigcup_{l=1}^{M_{J_i}} \{\alpha_{J_i,l},\beta_{J_i,l}\}}} \phi^{i,\tau, h}\discdep{u^{\tau,h}}{k,J} \nu_i^J.
\end{align*} And so \eqref{t2e1} implies
\begin{align*}
0&=\frac{1}{\tau}\sum_{{\boldsymbol x}_J \in \Omega^h} \left(\psi^{\tau, h}\discdep{u^{\tau,h}}{k,J}-\psi^{\tau, h}\discdep{u^{\tau,h}}{k-1,J}\right)+\frac{1}{h} \sum_{i=1}^n \sum_{\substack{{\boldsymbol x}_J \in \Omega^h\\ j_i\in\bigcup_{l=1}^{M_{J_i}} \{\alpha_{J_i,l},\beta_{J_i,l}\}}} \phi^{i,\tau, h}\discdep{u^{\tau,h}}{k,J} \nu_i^J \\
&=\frac{1}{\tau}\sum_{{\boldsymbol x}_J \in \Omega^h} \left(\psi^{\tau, h}\discdep{u^{\tau,h}}{k,J}-\psi^{\tau, h}\discdep{u^{\tau,h}}{k-1,J}\right)+\frac{1}{h} \sum_{{\boldsymbol x}_J \in \partial\Omega^h} \boldsymbol\phi^{\tau, h}\discdep{u^{\tau,h}}{k,J}\cdot \boldsymbol \nu^J,
\end{align*}where the last equality follows from definition of $\partial \Omega^h$ in \eqref{boundary}.
\end{proof}
{Although it may appear that Theorem \ref{scalarCL} restricts the discrete density and fluxes to be first order accurate, it is possible to obtain higher order accuracy for the density and fluxes, as we will illustrate in the examples.
}
\begin{cor} Let $\psi^{\tau,h}, \boldsymbol\phi^{\tau,h}$ be as given in Theorem \ref{scalarCL} and the discrete scalar function $u^{\tau,h}:\mesh\rightarrow \mathbb{R}$ be a solution to \eqref{discScalar} at $t_k$. If the discretized fluxes $\boldsymbol\phi^{\tau, h}\disc{u^{\tau,h}}$ vanish on the boundary $\partial\Omega^h$, then\begin{align*}
\sum_{{\boldsymbol x}_J \in \Omega^h}\psi^{\tau, h}\discdep{u^{\tau,h}}{k,J}=\sum_{{\boldsymbol x}_J \in \Omega^h}\psi^{\tau, h}\discdep{u^{\tau,h}}{k-1,J}.
\end{align*}
\end{cor} In other words, the proposed scheme has exact discrete conservation.
}

\subsection{System of $m$ PDEs with $m$ conservation laws}
\label{sec:squareSystem}

Next we extend the previous results to systems of PDEs. Consider a $k$-th order system of $m$ PDEs,
\rev{
\begin{equation}{\boldsymbol F}\contdep{\boldsymbol u}{t,\boldsymbol x}=0, \text{ on } t\in \mathcal{I}, {\boldsymbol x} \in \Omega, \label{systemPDE}\end{equation}
with $m$ conservation laws. In particular, we have an $m\times m$ matrix of conservation law multipliers $\Lambda\contdep{\boldsymbol u}{t,\boldsymbol x}$,
\begin{equation}\Lambda\contdep{\boldsymbol u}{t,\boldsymbol x}=\begin{pmatrix}
    \lambda^{11}\contdep{\boldsymbol u}{t,\boldsymbol x} & \dots & \lambda^{1m}\contdep{\boldsymbol u}{t,\boldsymbol x}\\
    \vdots &  & \vdots\\
    \lambda^{m1}\contdep{\boldsymbol u}{t,\boldsymbol x} & \dots & \lambda^{mm}\contdep{\boldsymbol u}{t,\boldsymbol x}\\
  \end{pmatrix}, \label{bigLambda}\end{equation} with their associated densities $\boldsymbol\psi\contdep{\boldsymbol u}{t,\boldsymbol x}$ and fluxes $\Phi\contdep{\boldsymbol u}{t,\boldsymbol x}$,
\begin{align}\boldsymbol \psi\contdep{\boldsymbol u}{t,\boldsymbol x}=\begin{pmatrix}
    \psi^{1}\contdep{\boldsymbol u}{t,\boldsymbol x} \\
    \vdots \\
    \psi^{m}\contdep{\boldsymbol u}{t,\boldsymbol x} \\
\end{pmatrix}, & \text{     } \aw{\Phi\contdep{\boldsymbol u}{t,\boldsymbol x}=\begin{pmatrix}
    \phi^{11}\contdep{\boldsymbol u}{t,\boldsymbol x} & \dots & \phi^{1n}\contdep{\boldsymbol u}{t,\boldsymbol x}\\
    \vdots &  & \vdots\\
    \phi^{m1}\contdep{\boldsymbol u}{t,\boldsymbol x} & \dots & \phi^{mn}\contdep{\boldsymbol u}{t,\boldsymbol x}\\
\end{pmatrix}}, \label{bigPhi}
\end{align} satisfying,
\begin{equation*} \Lambda\contdep{\boldsymbol u}{t,\boldsymbol x} \cdot \boldsymbol F\contdep{\boldsymbol u}{t,\boldsymbol x} = D_t \boldsymbol\psi\contdep{\boldsymbol u}{t,\boldsymbol x}+ D_{\boldsymbol x}\cdot \Phi\contdep{\boldsymbol u}{t,\boldsymbol x}.\end{equation*} \aw{Here the notation $D_{\boldsymbol x} \cdot \Phi$ is analogous to the divergence applied to the tensor $\Phi$, i.e. $(D_{\boldsymbol x} \cdot \Phi)^j = \sum_{j=1}^n D_{x_i}\phi^{ji}$.}

\noindent Analogous to \eqref{discScalarF}, we now show the following discretization of the system of PDEs is consistent provided the discrete multiplier matrix $\Lambda^{\tau,h}\discdep{\boldsymbol u^{\tau,h}}{k,J}$ is invertible on $\mesh$ and $\displaystyle \norm{\norm{\left({\Lambda^{\tau,h}}\right)^{-1}\disc{\boldsymbol u^{\tau,h}}}_{op}}_{l^\infty(\mesh)}$ is bounded above uniformly in $\tau, h$, where $\norm{\cdot}_{op}$ is the induced matrix norm of the vector norm $|\cdot|$.

\begin{thm} Let $\boldsymbol u \in C^{\max(p,q)+\max(r,k)}(\overline{\mathcal{I}\times \Omega})$ and $\Lambda$ in \eqref{bigLambda} be an $m\times m$ matrix of conservation law multipliers of \eqref{systemPDE} with corresponding density $\boldsymbol \psi$ and fluxes $\Phi$, as defined in \eqref{bigPhi}. Suppose $\Lambda^{\tau, h}$, $D_t^{\tau, h}\boldsymbol \psi$ and $D_{\boldsymbol x}^{\tau, h}\cdot \Phi$ are finite difference discretizations of $\Lambda$, $D_t\boldsymbol \psi$ and $D_{\boldsymbol x}\cdot \Phi$ with accuracy up to order $p$ in space and $q$ in time:
\begin{align*} 
\norm{ \norm{\Lambda\cont{\boldsymbol u}-\Lambda^{\tau, h}\disc{\boldsymbol u}}_{op} }_{l^\infty(\mesh)} & \leq C_\Lambda(h^p+\tau^q), \\
\norm{D_t\boldsymbol\psi\cont{\boldsymbol u}-D_t^{\tau, h}\boldsymbol\psi\disc{\boldsymbol u}}_{l^\infty(\mesh)} & \leq C_t(h^p+\tau^q), \\
\norm{D_{\boldsymbol x}\cdot \Phi\cont{\boldsymbol u}-D_{\boldsymbol x}^{\tau, h}\cdot \Phi\disc{\boldsymbol u}}_{l^\infty(\mesh)} & \leq C_x(h^p+\tau^q).
\end{align*} 
Let $\boldsymbol u^{\tau,h}:\mesh\rightarrow \mathbb{R}^m$ be a discrete function. Suppose $\Lambda^{\tau,h}\disc{\boldsymbol u^{\tau,h}}$ is invertible on $\mesh$ and that $\norm{\norm{\left({\Lambda^{\tau,h}}\right)^{-1}\disc{\boldsymbol u^{\tau,h}}}_{op}}_{l^\infty(\mesh)}\leq \gamma<\infty$, with $\gamma$ independent of $\tau, h$.

Define the finite difference discretization of $\boldsymbol F$ as:
\begin{equation} \boldsymbol F^{\tau, h}\discdep{\boldsymbol u^{\tau,h}}:=\left({\Lambda^{\tau,h}}\right)^{-1}\disc{\boldsymbol u^{\tau,h}}\left(D_t^{\tau, h}\boldsymbol\psi\disc{\boldsymbol u^{\tau,h}}+D_{\boldsymbol x}^{\tau, h}\cdot \Phi\disc{\boldsymbol u^{\tau,h}}\right). \label{discSystemF}\end{equation} Then there exists a constant $C>0$ independent of $h,\tau$ such that,
\begin{equation*} \norm{\boldsymbol F\cont{\boldsymbol u}-\boldsymbol F^{\tau, h}\disc{\boldsymbol u}}_{l^\infty(\mesh)} \leq C(h^p+\tau^q).\end{equation*}
\label{systemThm}
\end{thm}
\begin{proof}
The proof is nearly identical to the scalar case. 
Fix $(t_k,\boldsymbol x_J)\in \mesh$. Since $\Lambda$ is a conservation law $m\times m$ multiplier matrix of $\boldsymbol F$ and by definition of $\boldsymbol F^{\tau, h}$,
\begin{align}
\boldsymbol F\cont{\boldsymbol u}&-\boldsymbol F^{\tau,h}\disc{\boldsymbol u} \nonumber \\
&= \left({\Lambda^{\tau,h}}\right)^{-1}\disc{\boldsymbol u^{\tau,h}}\left( \Lambda^{\tau, h}\disc{\boldsymbol u}\boldsymbol F\cont{\boldsymbol u}-D_t^{\tau, h}\boldsymbol \psi\disc{\boldsymbol u}-D_{\boldsymbol x}^{\tau, h}\cdot \boldsymbol\Phi\disc{\boldsymbol u}\right) \nonumber \\
&=\left({\Lambda^{\tau,h}}\right)^{-1}\disc{\boldsymbol u^{\tau,h}}\left( (\Lambda^{\tau, h}\disc{\boldsymbol u}-\Lambda\cont{\boldsymbol u})\boldsymbol F\cont{\boldsymbol u}+\left(D_t \boldsymbol \psi\cont{\boldsymbol u}-D_t^{\tau, h}\boldsymbol \Psi\disc{\boldsymbol u}\right)+\left(D_{\boldsymbol x}\cdot \boldsymbol\Phi\cont{\boldsymbol u}-D_{\boldsymbol x}^{\tau, h}\cdot \boldsymbol\Phi\disc{\boldsymbol u}\right)\right) \label{t4e1}
\end{align}
Since $\boldsymbol u\in C^{\max(p,q)+\max(r,k)}(\overline{\mathcal{I}\times \Omega})$ and $\boldsymbol F[\boldsymbol u]$ is continuous in its arguments, $\boldsymbol F\contdep{\boldsymbol u}{t,\boldsymbol x}$ is continuous on $\overline{\mathcal{I}\times \Omega}$ and thus $|\boldsymbol F\contdep{\boldsymbol u}{t,\boldsymbol x}|\leq C'$ uniformly on $\overline{\mathcal{I}\times \Omega}$. By \eqref{t4e1},
\begin{align}
&\left| \boldsymbol F\cont{\boldsymbol u}-\boldsymbol F^{\tau,h}\disc{\boldsymbol u}\right| \nonumber\\
&\quad\quad\leq \norm{\left({\Lambda^{\tau,h}}\right)^{-1}\disc{\boldsymbol u^{\tau,h}}}_{op}\left( \left|\Lambda^{\tau, h}\disc{\boldsymbol u}-\Lambda\cont{\boldsymbol u}\right| |\boldsymbol F\cont{\boldsymbol u}|+\left|D_t \boldsymbol \psi\cont{\boldsymbol u}-D_t^{\tau, h}\boldsymbol \psi\disc{\boldsymbol u}\right|+\left|D_{\boldsymbol x}\cdot \boldsymbol\Phi\cont{\boldsymbol u}-D_{\boldsymbol x}^{\tau, h}\cdot \boldsymbol\Phi\disc{\boldsymbol u}\right|\right) \nonumber \\
&\quad\quad\leq \gamma(C'C_{\Lambda}(h^p+\tau^q))+C_{t}(h^p+\tau^q)+C_x(h^p+\tau^q)) = C(h^p+\tau^q), \label{t4e2}
\end{align} where $\displaystyle C:=\gamma\left(C'C_{\Lambda}+C_{t}+C_{x}\right)$. Since this is true for arbitrary mesh points $(t_k,\boldsymbol x_J)\in \mesh$, taking the maximum of \eqref{t4e2} over all points in $\mesh$ gives the desired result.
\end{proof}
As in the scalar case, we will return in Section \ref{sec:singularMult} to address the issue of when $\Lambda^{\tau, h}\disc{\boldsymbol u^{\tau,h}}$ become singular. The analog of conservative discretizations of $\boldsymbol F^{\tau, h}$ also holds for system of PDEs.
\begin{thm}
Let $\Lambda$ in \eqref{bigLambda} be an $m\times m$ matrix of conservation law multipliers  of \eqref{systemPDE} with density $\boldsymbol \psi$ and fluxes $\Phi$, as defined \eqref{bigPhi}. Let $\Lambda^{\tau, h}$ be a finite difference discretizations of $\Lambda$. Let $\boldsymbol u^{\tau,h}:\mesh\rightarrow \mathbb{R}^m$ be a discrete function. Define $D_t^{\tau, h}\boldsymbol \psi$ and $D_{\boldsymbol x}^{\tau, h}\cdot \Phi$ to be the following discretizations,
\begin{subequations}
\begin{align}
D_t^{\tau, h}\boldsymbol \psi\discdep{\boldsymbol u^{\tau,h}}{k,J} &= \frac{\boldsymbol\psi^{\tau,h}\discdep{\boldsymbol u^{\tau,h}}{k,J}-\boldsymbol\psi^{\tau,h}\discdep{\boldsymbol u^{\tau,h}}{k-1,J}}{\tau}, \label{dtSystem} \\
\left(D_{\boldsymbol x}^{\tau, h}\cdot \Phi \right)^j\discdep{\boldsymbol u^{\tau,h}}{k,J} &= \sum_{i=1}^n\frac{\phi^{j i,\tau,h}\discdep{\boldsymbol u^{\tau,h}}{k,J}-\phi^{j i,\tau,h}\discdep{\boldsymbol u^{\tau,h}}{k,J-\hat{i}}}{h},
\label{dxSystem}\end{align}
\end{subequations}where $\boldsymbol\psi^{\tau, h}$ and $\phi^{j i,\tau, h}$ are finite difference discretizations of $\boldsymbol\psi$ and $\phi^{j i}$. Also, let $\boldsymbol F^{\tau, h}$ be the corresponding discretization of \eqref{systemPDE} given by \eqref{discSystemF}. If $\boldsymbol u^{\tau,h}$ is a solution to the discrete problem \begin{equation}\boldsymbol F^{\tau, h}\discdep{\boldsymbol u^{\tau,h}}{k,J}=0,\hspace{3mm} {\boldsymbol x}_J\in \Omega^h, \label{discSystem}\end{equation} then the discrete divergence theorem holds,
\begin{align}
\boldsymbol 0=\frac{1}{\tau}\sum_{{\boldsymbol x}_J \in \Omega^h} \left(\boldsymbol\psi^{\tau,h}\discdep{\boldsymbol u^{\tau,h}}{k,J}-\boldsymbol\psi^{\tau,h}\discdep{\boldsymbol u^{\tau,h}}{k-1,J}\right)+\frac{1}{h} \sum_{{\boldsymbol x}_J \in \partial\Omega^h} \Phi^{\tau,h}\discdep{\boldsymbol u^{\tau,h}}{k,J} \cdot \boldsymbol \nu^J, \label{discDivThmSystem}
\end{align}where $\boldsymbol \nu^J=(\nu_1^J,\dots,\nu_n^J)$ is the outward-pointing unit normal of $\partial\Omega^h$ at $\boldsymbol x_J$. \label{vectorCL}
\end{thm}
\begin{proof} If $\boldsymbol u^{\tau,h}$ satisfies \eqref{discSystem}, then on each ${\boldsymbol x}_J\in \Omega^h$
\begin{equation}
\boldsymbol 0 = \Lambda^{\tau, h}\disc{\boldsymbol u^{\tau,h}}\boldsymbol F^{\tau, h}\disc{\boldsymbol u^{\tau,h}} = D_t^{\tau, h} {\boldsymbol \psi}\disc{\boldsymbol u^{\tau,h}} + D_{\boldsymbol x}^{\tau, h}\cdot \Phi\disc{\boldsymbol u^{\tau,h}}. \label{t5e1}
\end{equation} Now apply the same argument as in the scalar case to each component of \eqref{t5e1}.
\end{proof} {Similar to the scalar case, we mention that Theorem \ref{vectorCL} need not imply the discrete density and fluxes are restricted to be first order accurate. In the examples, we illustrate that it is possible to have higher order accuracy for the density and fluxes as well.
}
\begin{cor} Let $\boldsymbol\psi^{\tau,h}, \Phi^{\tau,h}$ be as given in Theorem \ref{vectorCL} and $\boldsymbol u^{\tau,h}$ be a solution to \eqref{discSystem}. If the discretized fluxes $\Phi^{\tau,h}$ vanish with $\boldsymbol u^{\tau,h}$ on the boundary $\partial\Omega^h$, then
\begin{align*}
\sum_{{\boldsymbol x}_J \in \Omega^h}\boldsymbol\psi^{\tau,h}\discdep{\boldsymbol u^{\tau,h}}{k,J}=\sum_{{\boldsymbol x}_J \in \Omega^h}\boldsymbol\psi^{\tau,h}\discdep{\boldsymbol u^{\tau,h}}{k-1,J}.
\end{align*}
\end{cor}In other words, the proposed scheme for systems also has exact discrete conservation.
}
\subsection{System of $m$ PDEs with $s$ conservation laws, $0<s<m$}
\label{sec:rectSystem}
\rev{
So far, we have assumed the number of equations $m$ is equal to the number of conservation laws $s$. We now extend our result to the case when $s<m$. We proceed as before, consider a $k$-th order system of $m$ PDEs \eqref{systemPDE} but now with $s$ conservation laws $\Lambda$,
\begin{equation}\Lambda\cont{\boldsymbol u}=\begin{pmatrix}
\widetilde{
\Lambda}\cont{\boldsymbol u} & \Sigma\cont{\boldsymbol u}
  \end{pmatrix},  \label{rectLambda}\end{equation}where $\widetilde{
\Lambda}$ is the $s\times s$ submatrix of $\Lambda$ and $
\Sigma$ is the $s\times (m-s)$ submatrix of $\Lambda$,

\begin{align}\widetilde{\Lambda}\cont{\boldsymbol u}=\begin{pmatrix}
    \lambda^{11}\cont{\boldsymbol u} & \dots & \lambda^{1s}\cont{\boldsymbol u}\\
    \vdots &  & \vdots\\
    \lambda^{s1}\cont{\boldsymbol u} & \dots & \lambda^{ss}\cont{\boldsymbol u}\\
  \end{pmatrix}, & \text{     } \Sigma\cont{\boldsymbol u} = \begin{pmatrix}
    \lambda^{1\hspace{1mm}s+1}\cont{\boldsymbol u} & \dots & \lambda^{1m}\cont{\boldsymbol u}\\
    \vdots &  & \vdots\\
    \lambda^{s\hspace{1mm}s+1}\cont{\boldsymbol u} & \dots & \lambda^{sm}\cont{\boldsymbol u}\\
  \end{pmatrix}.\nonumber \end{align}
By partitioning $\boldsymbol F$ in a similar manner,
\begin{equation*}
\boldsymbol F\cont{\boldsymbol u}=\begin{pmatrix}
\widetilde{\boldsymbol F}\cont{\boldsymbol u} \\ {\boldsymbol G\cont{\boldsymbol u}}
  \end{pmatrix},
\end{equation*} where $\widetilde{\boldsymbol F}$ are the first $s$ entries of $\boldsymbol F$ and $\boldsymbol G$ are the last $m-s$ entries of $\boldsymbol F$, one can write
\begin{equation}
D_t\boldsymbol \psi\cont{\boldsymbol u} + D_{\boldsymbol x}\cdot \Phi\cont{\boldsymbol u} =\Lambda\cont{\boldsymbol u} \boldsymbol F\cont{\boldsymbol u} = \begin{pmatrix}
\widetilde{
\Lambda}\cont{\boldsymbol u} & \Sigma\cont{\boldsymbol u}
  \end{pmatrix}\begin{pmatrix}
\widetilde{\boldsymbol F}\cont{\boldsymbol u} \\ {\boldsymbol G}\cont{\boldsymbol u}
  \end{pmatrix} = \widetilde{
\Lambda}\cont{\boldsymbol u}\widetilde{\boldsymbol F}\cont{\boldsymbol u} + \Sigma\cont{\boldsymbol u}{\boldsymbol G}\cont{\boldsymbol u}. \label{partitionExp}
\end{equation} 
Now if the row vectors of $\Lambda\contdep{\boldsymbol u}{t,\boldsymbol x}$ are locally linearly independent and hence (upon reordering of the equations of $\boldsymbol F\contdep{\boldsymbol u}{t,\boldsymbol x}$ if necessary) $\widetilde{\Lambda}^{-1}\contdep{\boldsymbol u}{t,\boldsymbol x}$ exists, the main idea is to rewrite $\widetilde{\boldsymbol F}\contdep{\boldsymbol u}{t,\boldsymbol x}$ as,
\begin{equation}
\widetilde{\boldsymbol F}\contdep{\boldsymbol u}{t,\boldsymbol x} = \widetilde{\Lambda}^{-1}\contdep{\boldsymbol u}{t,\boldsymbol x}\left(D_t\boldsymbol \psi\contdep{\boldsymbol u}{t,\boldsymbol x} + D_{\boldsymbol x}\cdot \Phi\contdep{\boldsymbol u}{t,\boldsymbol x} - \Sigma{\boldsymbol G}\contdep{\boldsymbol u}{t,\boldsymbol x}\right), \label{tildeF}
\end{equation} and discretizing both \eqref{tildeF} and $\boldsymbol G$. This leads to the following generalization of Theorem \ref{systemThm}.

\begin{thm} Let $\boldsymbol u \in C^{\max(p,q)+\max(r,k)}(\overline{\mathcal{I}\times \Omega})$ and $\Lambda$ be as defined in \eqref{rectLambda}. Suppose $\widetilde{\Lambda}^{\tau, h}$, $D_t^{\tau, h}\boldsymbol \psi$, $D_{\boldsymbol x}^{\tau, h}\cdot \Phi$, $\Sigma^{\tau,h}$ and ${\boldsymbol G}^{\tau,h}$ are finite difference discretizations of $\widetilde{\Lambda}$, $D_t\boldsymbol \psi$, $D_{\boldsymbol x}\cdot \Phi$,  $\Sigma$ and $\boldsymbol G$ with accuracy up to order $p$ in space and $q$ in time:
\begin{align*} 
\norm{ \norm{\widetilde{\Lambda}\cont{\boldsymbol u}-\widetilde{\Lambda}^{\tau, h}\disc{\boldsymbol u}}_{op} }_{l^\infty(\mesh)} & \leq C_{\widetilde{\Lambda}}(h^p+\tau^q), \\
\norm{D_t\boldsymbol\psi\cont{\boldsymbol u}-D_t^{\tau, h}\boldsymbol\psi\disc{\boldsymbol u}}_{l^\infty(\mesh)} & \leq C_t(h^p+\tau^q), \\
\norm{D_{\boldsymbol x}\cdot \Phi\cont{\boldsymbol u}-D_{\boldsymbol x}^{\tau, h}\cdot \Phi\disc{\boldsymbol u}}_{l^\infty(\mesh)} & \leq C_x(h^p+\tau^q),\\
\norm{ \norm{\Sigma \cont{\boldsymbol u}-\Sigma^{\tau, h}\disc{\boldsymbol u}}_{op} }_{l^\infty(\mesh)} & \leq C_\Sigma(h^p+\tau^q), \\
\norm{\boldsymbol G \cont{\boldsymbol u}-{\boldsymbol G}^{\tau,h}\disc{\boldsymbol u}}_{l^\infty(\mesh)} & \leq C_G(h^p+\tau^q).
\end{align*} 
Let $\boldsymbol u^{\tau,h}:\mesh\rightarrow \mathbb{R}^m$ be a discrete function. Suppose $\widetilde{\Lambda}^{\tau,h}\disc{\boldsymbol u^{\tau,h}}$ is invertible on $\mesh$ and that $\norm{\norm{\left({\widetilde{\Lambda}^{\tau,h}}\right)^{-1}\disc{\boldsymbol u^{\tau,h}}}_{op}}_{l^\infty(\mesh)}\leq \gamma<\infty$, with $\gamma$ independent of $\tau, h$.
Define the finite difference discretization of $\boldsymbol F$ as:
\begin{equation} {\boldsymbol F}^{\tau, h}\disc{\boldsymbol u^{\tau,h}}:=\begin{pmatrix}
\widetilde{\boldsymbol F}^{\tau, h}\disc{\boldsymbol u^{\tau,h}} \\ {\boldsymbol G^{\tau, h}}\disc{\boldsymbol u^{\tau,h}}
  \end{pmatrix}=
    \begin{pmatrix}\left({\widetilde{\Lambda}^{\tau,h}}\right)^{-1}\disc{\boldsymbol u^{\tau,h}}\left(D_t^{\tau, h} {\boldsymbol \psi}\disc{\boldsymbol u^{\tau,h}} + D_{\boldsymbol x}^{\tau, h}\cdot \Phi\disc{\boldsymbol u^{\tau,h}} - \Sigma^{\tau, h}\disc{\boldsymbol u^{\tau,h}}{\boldsymbol G}^{\tau, h}\disc{\boldsymbol u^{\tau,h}}\right)\\ {\boldsymbol G}^{\tau, h}\disc{\boldsymbol u^{\tau,h}}\end{pmatrix}. \label{discTildeF}\end{equation} 
Then there exists a constant $\widetilde{C}>0$ independent of $h,\tau$ such that for sufficiently small $h,\tau$,
\begin{equation*} \norm{\widetilde{\boldsymbol F}\cont{\boldsymbol u^{\tau,h}}-\widetilde{\boldsymbol F}^{\tau, h}\disc{\boldsymbol u^{\tau,h}}}_{l^\infty(\mesh)} \leq \widetilde{C}(h^p+\tau^q).\end{equation*}And hence of course, there exists a constant $C>0$ independent of $h,\tau$ such that for sufficiently small $h,\tau$,
 \begin{equation*} \norm{\boldsymbol F\cont{\boldsymbol u^{\tau,h}}-\boldsymbol F^{\tau, h}\disc{\boldsymbol u^{\tau,h}}}_{l^\infty(\mesh)} \leq C(h^p+\tau^q).\end{equation*}\label{augSystemThm}
\end{thm}
\begin{proof} The proof is nearly the same as in the square multiplier matrix case. \\
Fix $(t_k,\boldsymbol x_J)\in \mesh$. Since by \eqref{discTildeF} and \eqref{partitionExp},
\begin{align}
\widetilde{\boldsymbol F}\cont{\boldsymbol u}&-\widetilde{\boldsymbol F}^{\tau,h}\disc{\boldsymbol u} \nonumber \\
&= \left({\widetilde{\Lambda}^{\tau,h}}\right)^{-1}\disc{\boldsymbol u^{\tau,h}}\left( \widetilde{\Lambda}^{\tau, h}\disc{\boldsymbol u}\widetilde{\boldsymbol F}\cont{\boldsymbol u}-D_t^{\tau, h}\boldsymbol \psi\disc{\boldsymbol u}-D_{\boldsymbol x}^{\tau, h}\cdot \boldsymbol\Phi\disc{\boldsymbol u}+ \Sigma^{\tau, h}\disc{\boldsymbol u^{\tau,h}}{\boldsymbol G}^{\tau, h}\disc{\boldsymbol u^{\tau,h}}\right) \nonumber \\
&=\left({\widetilde{\Lambda}^{\tau,h}}\right)^{-1}\disc{\boldsymbol u^{\tau,h}}\left( \left(\widetilde{\Lambda}^{\tau, h}\disc{\boldsymbol u}-\widetilde{\Lambda}\cont{\boldsymbol u}\right)\widetilde{\boldsymbol F}\cont{\boldsymbol u}+\left(D_t\boldsymbol \psi\cont{\boldsymbol u}-D_t^{\tau, h}\boldsymbol \psi\disc{\boldsymbol u}\right)+ \left(D_{\boldsymbol x}\cdot \Phi\cont{\boldsymbol u} -D_{\boldsymbol x}^{\tau, h}\cdot \boldsymbol\Phi\disc{\boldsymbol u}\right)\right. \nonumber\\&\left. \hspace{25mm}+\Sigma^{\tau, h}\disc{\boldsymbol u^{\tau,h}}{\boldsymbol G}^{\tau, h}\disc{\boldsymbol u^{\tau,h}}-\Sigma\cont{\boldsymbol u}{\boldsymbol G}\cont{\boldsymbol u}\right)\nonumber\\
&=\left({\widetilde{\Lambda}^{\tau,h}}\right)^{-1}\disc{\boldsymbol u^{\tau,h}}\left( \left(\widetilde{\Lambda}^{\tau, h}\disc{\boldsymbol u}-\widetilde{\Lambda}\cont{\boldsymbol u}\right)\widetilde{\boldsymbol F}\cont{\boldsymbol u}+\left(D_t\boldsymbol \psi\cont{\boldsymbol u}-D_t^{\tau, h}\boldsymbol \psi\disc{\boldsymbol u}\right)+ \left(D_{\boldsymbol x}\cdot \Phi\cont{\boldsymbol u} -D_{\boldsymbol x}^{\tau, h}\cdot \boldsymbol\Phi\disc{\boldsymbol u}\right)\right. \nonumber\\&\left. \hspace{25mm}+\left(\Sigma^{\tau, h}\disc{\boldsymbol u^{\tau,h}}-\Sigma\cont{\boldsymbol u}\right){\boldsymbol G}^{\tau, h}\disc{\boldsymbol u^{\tau,h}}+\Sigma\cont{\boldsymbol u}\left({\boldsymbol G}^{\tau, h}\disc{\boldsymbol u^{\tau,h}}-{\boldsymbol G}\cont{\boldsymbol u}\right)\right).
 \label{t6e1}
\end{align}
Since $\boldsymbol u\in C^{\max(p,q)+\max(r,k)}(\overline{\mathcal{I}\times \Omega})$ and $\widetilde{\boldsymbol F}\cont{\boldsymbol u}, \Sigma\cont{\boldsymbol u}, {\boldsymbol G}\cont{\boldsymbol u}$ are continuous in their arguments, $\widetilde{\boldsymbol F}\contdep{\boldsymbol u}{t,\boldsymbol x}, \Sigma\contdep{\boldsymbol u}{t,\boldsymbol x}, {\boldsymbol G}\contdep{\boldsymbol u}{t,\boldsymbol x}$ are continuous on $\overline{\mathcal{I}\times \Omega}$. Thus $|\widetilde{\boldsymbol F}\contdep{\boldsymbol u}{t,\boldsymbol x}|\leq C_1$, $|\Sigma\contdep{\boldsymbol u}{t,\boldsymbol x}|\leq C_2$ and $|{\boldsymbol G}\contdep{\boldsymbol u}{t,\boldsymbol x}|\leq C_3$ uniformly on $\overline{\mathcal{I}\times \Omega}$. Note also for sufficiently small $h,\tau$,
\begin{equation}
|{\boldsymbol G}^{\tau, h}\disc{\boldsymbol u^{\tau,h}}|\leq |{\boldsymbol G}^{\tau, h}\disc{\boldsymbol u^{\tau,h}}-{\boldsymbol G}\cont{\boldsymbol u}|+|{\boldsymbol G}\cont{\boldsymbol u}| \leq C_G(h^p+\tau^q)+C_3 \leq 2C_3. \label{t6e2}
\end{equation}
So combining \eqref{t6e1} with \eqref{t6e2},
\begin{align}
&\left| \boldsymbol F\cont{\boldsymbol u}-\boldsymbol F^{\tau,h}\disc{\boldsymbol u}\right| \nonumber\\
&\quad\quad\leq \norm{\left({\widetilde{\Lambda}^{\tau,h}}\right)^{-1}\disc{\boldsymbol u^{\tau,h}}}_{op}\left( \left|\widetilde{\Lambda}^{\tau, h}\disc{\boldsymbol u}-\widetilde{\Lambda}\cont{\boldsymbol u}\right|\left|\widetilde{\boldsymbol F}\cont{\boldsymbol u}\right|+\left|D_t\boldsymbol \psi\cont{\boldsymbol u}-D_t^{\tau, h}\boldsymbol \psi\disc{\boldsymbol u}\right|+ \left|D_{\boldsymbol x}\cdot \Phi\cont{\boldsymbol u} -D_{\boldsymbol x}^{\tau, h}\cdot \boldsymbol\Phi\disc{\boldsymbol u}\right|\right. \nonumber\\&\left. \hspace{40mm}+\left|\Sigma^{\tau, h}\disc{\boldsymbol u^{\tau,h}}-\Sigma\cont{\boldsymbol u}\right|\left|{\boldsymbol G}^{\tau, h}\disc{\boldsymbol u^{\tau,h}}\right|+\left|\Sigma\cont{\boldsymbol u}\right|\left|{\boldsymbol G}^{\tau, h}\disc{\boldsymbol u^{\tau,h}}-{\boldsymbol G}\cont{\boldsymbol u}\right|\right)
 \nonumber \\
&\quad\quad\leq \gamma\left(\frac{}{}C_1C_{\Lambda}(h^p+\tau^q)+C_{t}(h^p+\tau^q)+C_x(h^p+\tau^q)+2C_3C_\Sigma(h^p+\tau^q)+C_2C_G(h^p+\tau^q)\right) \nonumber\\
&\quad\quad= \widetilde{C}(h^p+\tau^q), \label{t6e2}
\end{align} where $\displaystyle \widetilde{C}:=\gamma\left(C_1C_{\Lambda}+C_{t}+C_{x}+2C_3C_\Sigma+C_2C_G\right)$. Since this is true for arbitrary mesh points $(t_k,\boldsymbol x_J)\in \mesh$, taking the maximum of \eqref{t6e2} over all points in $\mesh$ gives the desired result.
\end{proof}
By augmenting $\widetilde{\boldsymbol F}^{\tau,h}$ with ${\boldsymbol G}^{\tau, h}$, we also obtain the desired conservative property.
\begin{thm}
Let $\widetilde{\Lambda}^{\tau, h}, \Sigma^{\tau,h}, \boldsymbol G^{\tau, h}$ be finite difference discretizations of $\widetilde{\Lambda}, \Sigma, \boldsymbol G$. Let $\boldsymbol u^{\tau,h}:\mesh\rightarrow \mathbb{R}^m$ be a discrete function. Let $D_t^{\tau, h}\boldsymbol \psi$ and $D_{\boldsymbol x}^{\tau, h}\cdot \Phi$ be the following discretizations,
\begin{subequations}
\begin{align}
D_t^{\tau, h}\boldsymbol \psi\discdep{\boldsymbol u^{\tau,h}}{k,J} &= \frac{\boldsymbol\psi^{\tau,h}\discdep{\boldsymbol u^{\tau,h}}{k,J}-\boldsymbol\psi^{\tau,h}\discdep{\boldsymbol u^{\tau,h}}{k-1,J}}{\tau},  \label{dtAugSystem} \\
\left(D_{\boldsymbol x}^{\tau, h}\cdot \Phi \right)^j\discdep{\boldsymbol u^{\tau,h}}{k,J} &= \sum_{i=1}^n\frac{\phi^{j i,\tau,h}\discdep{\boldsymbol u^{\tau,h}}{k,J}-\phi^{j i,\tau,h}\discdep{\boldsymbol u^{\tau,h}}{k,J-\hat{i}}}{h}, \label{dxAugSystem}
\end{align}\end{subequations}where $\boldsymbol\psi^{\tau, h}$ and $\phi^{j i,\tau, h}$ are finite difference discretizations of $\boldsymbol\psi$ and $\phi^{j i}$. Let $\boldsymbol F^{\tau, h}$ be the corresponding discretization of \eqref{systemPDE} given by \eqref{discTildeF}. If $\boldsymbol u^{\tau,h}$ is a solution to the discrete problem \begin{equation}\boldsymbol F^{\tau, h}\discdep{\boldsymbol u^{\tau,h}}{k,J}=\begin{pmatrix}
\widetilde{\boldsymbol F}^{\tau, h}\discdep{\boldsymbol u^{\tau,h}}{k,J} \\ {\boldsymbol G}^{\tau, h}\discdep{\boldsymbol u^{\tau,h}}{k,J}
  \end{pmatrix} = \boldsymbol 0, \hspace{3mm}{\boldsymbol x}_J\in \Omega^h, \label{discAugSystem}\end{equation} then the discrete divergence theorem \eqref{discDivThmSystem} holds for $\boldsymbol u^{\tau,h}$.
 \label{vectorAugCL}
\end{thm}
\begin{proof} If $\boldsymbol u^{\tau,h}$ satisfies \eqref{discAugSystem}, $\widetilde{\boldsymbol F}^{\tau, h}\discdep{\boldsymbol u^{\tau,h}}{k,J}={\boldsymbol 0}$ and ${\boldsymbol G}^{\tau, h}\discdep{\boldsymbol u^{\tau,h}}{k,J} = {\boldsymbol 0}$ on each ${\boldsymbol x}_J\in \Omega^h$. Hence,
\begin{align}
\boldsymbol 0 &= (\tilde{\Lambda}^{\tau, h}\discdep{\boldsymbol u^{\tau,h}}{k,J}) \widetilde{\boldsymbol F}^{\tau, h}\discdep{\boldsymbol u^{\tau,h}}{k,J} \nonumber \\
&= D_t^{\tau, h} {\boldsymbol \psi}\discdep{\boldsymbol u^{\tau,h}}{k,J} + D_{\boldsymbol x}^{\tau, h}\cdot \Phi\discdep{\boldsymbol u^{\tau,h}}{k,J} - \Sigma^{\tau, h}\discdep{\boldsymbol u^{\tau,h}}{k,J}{\boldsymbol G}^{\tau, h}\discdep{\boldsymbol u^{\tau,h}}{k,J} \nonumber\\
&=  D_t^{\tau, h} {\boldsymbol \psi}\discdep{\boldsymbol u^{\tau,h}}{k,J} + D_{\boldsymbol x}^{\tau, h}\cdot \Phi\discdep{\boldsymbol u^{\tau,h}}{k,J}. \label{t7e1}
\end{align}Now apply the same argument as in the scalar case to each component of \eqref{t7e1}.
\end{proof} 
\begin{cor} Let $\boldsymbol\psi^{\tau,h}, \Phi^{\tau,h}$ be as given in Theorem \ref{vectorAugCL} and $\boldsymbol u^{\tau,h}$ be a solution to \eqref{discAugSystem}. If the discretized fluxes $\Phi^{\tau,h}\disc{\boldsymbol u^{\tau,h}}$ vanish on the boundary $\partial\Omega^h$, then
\begin{align*}
\sum_{{\boldsymbol x}_J \in \Omega^h}\boldsymbol\psi^{\tau,h}\discdep{\boldsymbol u^{\tau,h}}{k,J}=\sum_{{\boldsymbol x}_J \in \Omega^h}\boldsymbol\psi^{\tau,h}\discdep{\boldsymbol u^{\tau,h}}{k-1,J}.
\end{align*}
\end{cor}\noindent Again, the proposed scheme in this case has exact discrete conservation.
}
\section{Examples}
\label{sec:examples}
\aw{We now illustrate our method with examples of the three cases discussed in the theory section. We begin with examples involving scalar and systems of ODEs, followed by scalar and systems of PDEs.}

\subsection{Pendulum problem}
\label{sec:pendProblem}
We first begin with the ODE for the pendulum problem,
\rev{\begin{equation*}
F\cont{\theta}:=\theta_{tt}+\frac{g}{l} \sin (\theta)=0, \label{pendEqn}
\end{equation*}}\noindent where $g$ is the gravitational acceleration, $l$ is the length of the pendulum arm and $\theta$ is the displacement angle of the pendulum. It is well-known from classical physics that the energy of such a system is conserved. In particular, we have one conservation law
\rev{\begin{equation*}
\left.D_t\left( \frac{1}{2}(\theta_t)^2 - \frac{g}{l} \cos(\theta)\right)\right|_{F\cont{\theta}=0}=\left.\lambda\cont{\theta}  F\cont{\theta}\frac{}{}\right|_{F\cont{\theta}=0}=0, \label{pendCL}
\end{equation*}}\noindent with the conservation law multiplier,
\rev{\begin{equation}
\lambda\cont{\theta}=\theta_t. \label{pendMult} \nonumber
\end{equation}}Note that we could have also found the multiplier and conservation law through the use of Euler operator.
Now suppose we discretize $\psi$, $\lambda$ with the following expressions,
\rev{\begin{align}
\psi^{\tau}\discdep{\theta^{\tau}}{n} &:= \frac{1}{2}\left(\frac{\theta_{n+1}-\theta_n}{\tau}\right)^2-\frac{g}{2l}\left(\cos(\theta_{n+1})+\cos(\theta_{n})\right), \label{pendPsi} \\
\lambda^{\tau}\discdep{\theta^{\tau}}{n}&:=\frac{\theta_{n+1}-\theta_{n-1}}{2\tau}. \label{pendDiscMult}
\end{align}}\noindent Then, $D_t^{\tau}\psi$ defined by \eqref{dtScalar} simplifies to,
\rev{\begin{align}D_t^{\tau}\psi\discdep{\theta^{\tau}}{n}&:
  = \frac{(\theta_{n+1}-\theta_{n-1})(\theta_{n+1}-2\theta_n+\theta_{n-1})}{2\tau^3}-\frac{g}{l}\frac{\cos(\theta_{n+1})-\cos(\theta_{n-1})}{2\tau}, \label{pendDPsi}
\end{align}}\noindent and so by \eqref{discScalarF}, we have the discretization of $F$,
\rev{\begin{align}F^{\tau}\discdep{\theta^{\tau}}{n}&:
  = \frac{\theta_{n+1}-2\theta_n+\theta_{n-1}}{\tau^2}-\frac{g}{l}\frac{\cos(\theta_{n+1})-\cos(\theta_{n-1})}{\theta_{n+1}-\theta_{n-1}}= 0. \label{pendDisc} 
\end{align}}\noindent Note that the discrete density in \eqref{pendDPsi} is actually second order accurate to $\psi$ in $\tau$, even though $D_t^{\tau}\psi$ defined by \eqref{dtScalar} may appear to be first order. Combining with the fact that $\lambda^{\tau}$ is also second order accurate, $F^{\tau}$ is second order accurate as well by the consistency Theorem \ref{scalarThm}. Multiplying \eqref{pendDisc} by \eqref{pendDiscMult} shows that the discrete density \eqref{pendPsi} is conserved for solutions of \eqref{pendDisc}, as claimed by Theorem \ref{scalarCL}.

\subsection{Damped harmonic oscillator}
\label{sec:dho}
Recall the damped harmonic oscillator from classical mechanics,
\rev{\begin{equation}
F\cont{x}:=mx_{tt}+\gamma x_t+k x=0, \label{dhoEqn}
\end{equation}}\noindent where $x(t)$ is the displacement of an object with mass $m$ attached to a spring with a spring constant $k$ and a damping coefficient $\gamma$. While the energy is not conserved due to dissipation, it can be found using the method of Euler operator that \eqref{dhoEqn} has the following non-trivial conservation law
\rev{\begin{equation}
\left.D_t\left( \frac{e^{\frac{\gamma}{m} t}}{2}\left(m\left(x_t+\frac{\gamma}{2m}x\right)^2+\left(k-\frac{\gamma^2}{4m}\right)x^2\right) \right)\right|_{F\cont{x}=0}=\left.\lambda\cont{x}  F\cont{x}\right|_{F\cont{x}=0}=0,  \label{dhoCL}
\end{equation}}\noindent with the conservation law multiplier,
\rev{\begin{equation}
\lambda\cont{x} = e^{\frac{\gamma t}{m}}\left(x_t+\frac{\gamma}{2m}x\right). \label{dhoMult}
\end{equation}}\noindent We have included the derivation for the conservation law \eqref{dhoCL} and multiplier \eqref{dhoMult} in Appendix A.\\
Note that the density function $\psi$ and multiplier $\lambda$ can be rewritten as 
\rev{$$\psi\cont{x} = \frac{m}{2}\left((e^{\frac{\gamma t}{2m}}x)_t\right)^2+\frac{1}{2}\left(k-\frac{\gamma^2}{4m}\right)(e^{\frac{\gamma t}{2m}}x)^2,$$
$$\lambda\cont{x} = e^{\frac{\gamma t}{2m}}(e^{\frac{\gamma t}{2m}}x)_t.$$}\noindent Thus choosing the following discretization for $\psi$ and $\lambda$,
\rev{\begin{align}
\psi^\tau\discdep{x^{\tau}}{n}&:=\frac{m}{2}\left(\frac{e^{\frac{\gamma t_{n+1}}{2m}}x_{n+1}-e^{\frac{\gamma t_{n}}{2m}}x_{n}}{\tau}\right)^2+\frac{1}{2}\left(k-\frac{\gamma^2}{4m}\right)\left(\frac{e^{\frac{\gamma t_{n+1}}{2m}}x_{n+1}+e^{\frac{\gamma t_{n}}{2m}}x_{n}}{2}\right)^2, \label{dhoPsi} \\
\lambda^\tau\discdep{x^{\tau}}{n} &:= e^{\frac{\gamma t_n}{2m}}\left(\frac{e^{\frac{\gamma t_{n+1}}{2m}}x_{n+1}-e^{\frac{\gamma t_{n-1}}{2m}}x_{n-1}}{2\tau}\right), \label{dhoDiscMult}
\end{align}}\noindent $D_t^\tau \psi$ defined by \eqref{dtScalar} simplifies to,
\rev{\begin{align}
D_t^\tau \psi\discdep{x^{\tau}}{n} =& m\left(\frac{e^{\frac{\gamma t_{n+1}}{2m}}x_{n+1}-2e^{\frac{\gamma t_{n}}{2m}}x_n+e^{\frac{\gamma t_{n-1}}{2m}}x_{n-1}}{\tau^2}\right)\left(\frac{e^{\frac{\gamma t_{n+1}}{2m}}x_{n+1}-e^{\frac{\gamma t_{n-1}}{2m}}x_{n-1}}{2\tau}\right) \nonumber \\
&+\left(k-\frac{\gamma^2}{4m}\right)\left(\frac{e^{\frac{\gamma t_{n+1}}{2m}}x_{n+1}+2e^{\frac{\gamma t_{n}}{2m}}x_n+e^{\frac{\gamma t_{n-1}}{2m}}x_{n-1}}{4}\right)\left(\frac{e^{\frac{\gamma t_{n+1}}{2m}}x_{n+1}-e^{\frac{\gamma t_{n-1}}{2m}}x_{n-1}}{2\tau}\right). \label{dhoDPsi}
\end{align}}\noindent By \eqref{discScalarF}, we have the discretization of $F$,
\rev{\begin{align}
F^{\tau}\discdep{x^{\tau}}{n}&:=m\left(\frac{e^{\frac{\gamma \tau}{2m}}x_{n+1}-2x_n+e^{-\frac{\gamma \tau}{2m}}x_{n-1}}{\tau^2}\right)+\left(k-\frac{\gamma^2}{4m}\right)\left(\frac{e^{\frac{\gamma \tau}{2m}}x_{n+1}+2x_n+e^{-\frac{\gamma \tau}{2m}}x_{n-1}}{4}\right)=0. \label{dhoDisc}
\end{align}}Indeed, it can be checked that the first term in \eqref{dhoDisc} approaches $mx_{tt}+\gamma x_t+\frac{\gamma^2}{4m}x$ as $\tau\rightarrow 0$ by l'H\^{o}pital's rule and the second term in \eqref{dhoDisc} approaches $\left(k-\frac{\gamma^2}{4m}\right)x$ as $\tau\rightarrow 0$. In other words, $F^\tau_n$ defined in \eqref{dhoDisc} is consistent with $F$ in \eqref{dhoEqn}, as claimed by Theorem \ref{scalarThm}. Moreover, {since $\lambda^\tau$ in \eqref{dhoDiscMult} and $D_t^\tau \psi$ in \eqref{dhoDPsi} are second order accurate, $F^{\tau}$ is second order accurate as well by Theorem \ref{scalarThm}.
}
Multiplying \eqref{dhoDisc} by \eqref{dhoDiscMult} shows that the discrete density \eqref{dhoPsi} is conserved for solutions of \eqref{dhoDisc}, as claimed by Theorem \ref{scalarCL}. {In Appendix B, we include numerical verification of the order of accuracy and of the exact conservation for the discrete density $\psi^\tau$ in \eqref{dhoPsi}.}

\subsection{Two body problem}
Consider the ODE system of the two-body problem in 1D arising from classical physics,
\rev{\begin{equation}
\boldsymbol F\cont{\boldsymbol x}:=\begin{bmatrix}
    x^1_{tt} - V'(x^1-x^2)\\
    x^2_{tt} - V'(x^2-x^1)\\
  \end{bmatrix}=\boldsymbol 0, \label{twoBodyEqn} \nonumber
\end{equation}}where $x^1, x^2$ are position of the two particles and $V$ is the interaction potential satisfying $V'(-z)=-V'(z)$. In this case, it's known from Noether's theorem that both momentum and energy is conserved. In particular, we have exactly two conservation laws
\rev{\begin{equation}
\left.D_t\begin{pmatrix}
    x^1_t + x^2_t\\
     \frac{(x^1_t)^2 + (x^2_t)^2}{2}+V(x^1-x^2)\\
  \end{pmatrix}\right|_{\boldsymbol F\cont{\boldsymbol x}=0}=\left.\Lambda\cont{\boldsymbol x} \boldsymbol F\cont{\boldsymbol x}\frac{}{}\right|_{\boldsymbol F\cont{\boldsymbol x}=0}=0, \label{twoBodyCL} \nonumber
\end{equation}}\noindent and two sets of conservation law multipliers $\Lambda$,
\rev{\begin{equation}
\Lambda\cont{\boldsymbol x}=\begin{pmatrix}
    1 & 1\\
    x^1_t & x^2_t\\
  \end{pmatrix}, \label{twoBodyMult} \nonumber
\end{equation}}\noindent As in the previous two examples, we could have also found the multipliers and conservation laws through the method of Euler operator.
We discretize $\boldsymbol \psi$ and $\Lambda$ by,
\rev{\begin{align}\boldsymbol \psi^\tau\discdep{\boldsymbol x^{\tau}}{n}&: =\begin{pmatrix}
    \frac{x^1_{n+1}-x^1_{n}}{\tau}+ \frac{x^2_{n+1}-x^2_{n}}{\tau}\\
    \frac{1}{2}\left(\frac{x^1_{n+1}-x^1_{n}}{\tau}\right)^2+\frac{1}{2}\left(\frac{x^2_{n+1}-x^2_{n}}{\tau}\right)^2+ \frac{V(x^1_{n+1}-x^2_{n+1})+V(x^1_{n}-x^2_{n})}{2}\\
  \end{pmatrix}, \label{twoBodyPsi} \\
  \Lambda^{\tau}\discdep{\boldsymbol x^{\tau}}{n}&:=\begin{pmatrix}
    1 & 1\\
    \frac{x^1_{n+1}-x^1_{n-1}}{2\tau} & \frac{x^2_{n+1}-x^2_{n-1}}{2\tau}\\
  \end{pmatrix}. \label{twoBodyDiscMult}
\end{align}}\noindent Then $D_t^{\tau}\boldsymbol \psi$ given by \eqref{dtSystem} simplifies to
\rev{\begin{align}
D_t^{\tau}\boldsymbol \psi\discdep{\boldsymbol x^{\tau}}{n}&:
  =\begin{pmatrix}
    \frac{x^1_{n+1}-2x^1_{n}+x^1_{n-1}}{\tau^2}+ \frac{x^2_{n+1}-2x^2_{n}+x^2_{n-1}}{\tau^2}\\
    \frac{(x^1_{n+1}-x^1_{n-1})(x^1_{n+1}-2x^1_{n}+x^1_{n-1})}{2\tau^3}+\frac{(x^2_{n+1}-x^2_{n-1})(x^2_{n+1}-2x^2_{n}+x^2_{n-1})}{2\tau^3}+ \frac{V(x^1_{n+1}-x^2_{n+1})-V(x^1_{n-1}-x^2_{n-1})}{2\tau}\\
  \end{pmatrix}. \label{twoBodyDPsi}
\end{align}}\noindent Since $\Lambda^{\tau,h}$ in \eqref{twoBodyDiscMult} and $D_t^{\tau}\boldsymbol \psi$ in \eqref{twoBodyDPsi} are both second order accurate, we have by Theorem \ref{systemThm} that the second order accurate discretization $\boldsymbol F^\tau$,
\rev{\begin{equation}
\boldsymbol F^{\tau}\discdep{\boldsymbol x^{\tau}}{n}:=\begin{pmatrix}
    \frac{x^1_{n+1}-2x^1_{n}+x^1_{n-1}}{\tau^2}+\frac{V(x^1_{n+1}-x^2_{n+1})-V(x^1_{n-1}-x^2_{n-1})}{x^1_{n+1}-x^2_{n+1}-x^1_{n-1}+x^2_{n-1}}\\
    \frac{x^2_{n+1}-2x^2_{n}+x^2_{n-1}}{\tau^2}+\frac{V(x^2_{n+1}-x^1_{n+1})-V(x^2_{n-1}-x^1_{n-1})}{x^2_{n+1}-x^1_{n+1}-x^2_{n-1}+x^1_{n-1}}\\
  \end{pmatrix}=\boldsymbol 0. \label{twoBodyDisc}
\end{equation}}\noindent Multiplying \eqref{twoBodyDisc} by \eqref{twoBodyDiscMult} shows that the discrete densities \eqref{twoBodyPsi} are preserved.

\subsection{Lotka-Volterra equations}
Consider the Lotka-Volterra equations or the predator-prey equations,
\rev{\begin{equation}
\boldsymbol F\cont{\boldsymbol x}:=\begin{pmatrix}
    x_t - x(a-by)\\
    y_t + y(c-dx)\\
  \end{pmatrix}=\boldsymbol 0, \label{predPreyEqn} \nonumber
\end{equation}}where $a,b,c,d$ are positive constants. It is known that this system has one conservation law,
\rev{\begin{equation}
\left.D_t\left(\frac{}{}\log(x^c y^a)-dx-by\right)\right|_{\boldsymbol F\cont{\boldsymbol x}=0}=\left.\Lambda \cont{\boldsymbol x}\boldsymbol F\cont{\boldsymbol x}\frac{}{}\right|_{\boldsymbol F\cont{\boldsymbol x}=0}=0, \label{predPreyCL} \nonumber
\end{equation}}with the $1\times 2$ multiplier matrix $\Lambda$,
\rev{\begin{equation}
\Lambda\cont{\boldsymbol x}=\begin{pmatrix}
\frac{c}{x}-d & \frac{a}{y}-b\\
  \end{pmatrix}. \label{predPreyMult}
\end{equation}}In particular, this system is of the kind where the number of conservation laws is less than the number of equations. Hence, we first partition \eqref{predPreyMult} into $1\times 1$ matrices $\widetilde{\Lambda}$ and $\Sigma$,
\rev{\begin{align}
\widetilde{\Lambda}\cont{\boldsymbol x}=\begin{pmatrix}
\frac{c}{x}-d 
  \end{pmatrix}, \nonumber \\
\Sigma\cont{\boldsymbol x}=\begin{pmatrix}
\frac{a}{y}-b
  \end{pmatrix}, \nonumber
\end{align}}\noindent so that \rev{$G\cont{\boldsymbol x}=y_t + y(c-dx)$}. In order to use \eqref{discTildeF}, we discretize $\psi$, $\widetilde{\Lambda}$, $\Sigma$, $G$ by
\rev{\begin{align}\psi^\tau\discdep{\boldsymbol x^\tau}{n}&: = \log(x^c_n y^a_n)-dx_n-by_n,\label{predPreyPsi} \\
  \widetilde{\Lambda}^{\tau}\discdep{\boldsymbol x^\tau}{n}&:=\begin{pmatrix}
    \frac{c}{x_n}-d
  \end{pmatrix}, \nonumber\\
  \Sigma^{\tau}\discdep{\boldsymbol x^\tau}{n}&:=\begin{pmatrix}
    \frac{a}{y_n}-b
  \end{pmatrix}, \nonumber\\
  G^{\tau}\discdep{\boldsymbol x^\tau}{n}&:=  \frac{y_{n}-y_{n-1}}{\tau} + y_{n}(c-dx_{n}). \nonumber
\end{align}}\noindent Then $D_t^{\tau} \psi$ defined by \eqref{dtAugSystem} is, 
\rev{\begin{align*}D_t^{\tau}\psi\discdep{\boldsymbol x^\tau}{n}&:
  = \frac{1}{\tau}\left(\log\left(\frac{x^c_{n} y^a_{n}}{x^c_{n-1} y^a_{n-1}}\right)-d(x_{n}-x_{n-1})-b(y_{n}-y_{n-1})\right).
\end{align*}}{Although $\widetilde{\Lambda}^{\tau}$ and $\widetilde{\Sigma}^{\tau}$ are exact discretization of $\widetilde{\Lambda}$ and $\widetilde{\Sigma}$, $D_t^{\tau} \psi$ and $G^{\tau}$ are only first order accurate.
So the discretization ${\boldsymbol F}^\tau$ given by \eqref{discTildeF} is at most first order by Theorem \ref{augSystemThm} and it simplifies to,}
\rev{\begin{equation}
\boldsymbol F^\tau\discdep{\boldsymbol x^\tau}{n}:=\begin{pmatrix}\frac{1}{\frac{c}{x_n}-d}\left[c\frac{\log x_{n}-\log x_{n-1}}{\tau}-d\frac{x_{n}-x_{n-1}}{\tau}-\left(\frac{c}{x_n}-d\right)x_n(a-by_n)+a\frac{\log y_{n}-\log y_{n-1}}{\tau}-\frac{a}{y_n}\frac{y_{n}-y_{n-1}}{\tau}\right] \\ \frac{y_{n}-y_{n-1}}{\tau} + y_{n}(c-dx_{n})\end{pmatrix}=\boldsymbol 0.
 \label{predPreyDisc}
\end{equation}}It can be seen that by taking $\tau \rightarrow 0$, the first equation of $\boldsymbol F^{\tau}$ is consistent with the first equation of $\boldsymbol F$. Moreover, it can be checked that solutions of \eqref{predPreyDisc} preserve \eqref{predPreyPsi}.

\subsection{Non-dissipative Lorenz equations}
Consider the non-dissipative Lorenz equations,
\rev{\begin{equation*}
\boldsymbol F\cont{\boldsymbol x}:=\begin{pmatrix}
    x_t - \sigma y\\
    y_t - x(r-z)\\
    z_t - xy
  \end{pmatrix}=\boldsymbol 0
\end{equation*}}\noindent where $\sigma, r$ are positive constants. It was found in \cite{NevBle94} that this system has $2$ conservation laws, 
\rev{\begin{equation*}
\left.D_t\begin{pmatrix}
    z-\frac{x^2}{2\sigma}\\
     \frac{y^2}{2}+\frac{z^2}{2}-rz\\
  \end{pmatrix}\right|_{\boldsymbol F\cont{\boldsymbol x}=0}=\left.\Lambda\cont{\boldsymbol x} \boldsymbol F\cont{\boldsymbol x}\frac{}{}\right|_{\boldsymbol F\cont{\boldsymbol x}=0}=0,
\end{equation*}}\noindent with the $2\times 3$ multiplier matrix $\Lambda$,
\rev{\begin{equation}
\Lambda\cont{\boldsymbol x}=\begin{pmatrix}
-\frac{x}{\sigma} & 0 & 1\\
0 & y & z-r\\
  \end{pmatrix}. \label{lorenzMult}
\end{equation}}\noindent Similar to the previous example, we partition \eqref{lorenzMult} into $2 \times 2$ matrix $\widetilde{\Lambda}$ and $1 \times 2$ matrix $\Sigma$,
\rev{\begin{align}
\widetilde{\Lambda}\cont{\boldsymbol x}=\begin{pmatrix}
-\frac{x}{\sigma} & 0 \\
0 & y
  \end{pmatrix}, \nonumber \\
\Sigma\cont{\boldsymbol x}=\begin{pmatrix}
1 \\ z-r
  \end{pmatrix}, \nonumber
\end{align}}with \rev{$G\cont{\boldsymbol x} = z_t - xy$}. To simplify the final discretization as much as possible, we choose the following discretization $\psi$, $\widetilde{\Lambda}$, $\Sigma$, $G$ by
\rev{\begin{align}\boldsymbol \psi^\tau\discdep{\boldsymbol x^\tau}{n}&: = \begin{pmatrix}
z_n- \frac{x_n^2}{2\sigma} \\
\frac{y_n^2}{2}+\frac{z_n^2}{2}-rz_n
\end{pmatrix},\label{lorenzPsi} \\
  \widetilde{\Lambda}^{\tau}\discdep{\boldsymbol x^\tau}{n}&:=\begin{pmatrix}
    -\frac{x_n+x_{n-1}}{2\sigma} & 0\\
    0 & \frac{y_n+y_{n-1}}{2}
  \end{pmatrix}, \label{lorenzDiscMult1}\\
  \Sigma^{\tau}\discdep{\boldsymbol x^\tau}{n}&:=\begin{pmatrix}
    1 \\
    \frac{z_n+z_{n-1}}{2}-r
  \end{pmatrix}, \label{lorenzDiscMult2}\\
  G^{\tau}\discdep{\boldsymbol x^\tau}{n}&:=  \frac{z_{n}-z_{n-1}}{\tau}-\frac{x_{n}+x_{n-1}}{2}\frac{y_{n}+y_{n-1}}{2}. \nonumber
\end{align}}Then $D_t^{\tau} \psi$ defined by \eqref{dtAugSystem} is, 
\rev{\begin{align*}D_t^{\tau}\boldsymbol \psi\discdep{\boldsymbol x^\tau}{n}&:
  =\begin{pmatrix}
    \frac{z_{n}-z_{n-1}}{\tau}-\frac{x_{n}^2-x_{n-1}^2}{2\sigma\tau}\\
    \frac{y_{n}^2-y_{n-1}^2}{2\tau}+\frac{z_{n}^2-z_{n-1}^2}{2\tau} -r\frac{z_{n}-z_{n-1}}{\tau}
  \end{pmatrix}.
\end{align*}}{Since $\widetilde{\Lambda}^{\tau}$ and $\widetilde{\Sigma}^{\tau}$ are second order accurate and $D_t^{\tau} \psi$ and $G^{\tau}$ are only first order accurate, we expect the discretization to be at most first order by Theorem \ref{augSystemThm}.}
Using \eqref{discTildeF}, the discretization ${\boldsymbol F}^\tau$ simplifies to,
\rev{\begin{equation}
\boldsymbol F^\tau\discdep{\boldsymbol x^\tau}{n}:=\begin{pmatrix}
\frac{x_{n}-x_{n-1}}{\tau}-\sigma\frac{y_{n}+y_{n-1}}{2}\\
\frac{y_{n}-y_{n-1}}{\tau}-\frac{x_{n}+x_{n-1}}{2}(r-\frac{z_{n}+z_{n-1}}{2})\\
\frac{z_{n}-z_{n-1}}{\tau}-\frac{x_{n}+x_{n-1}}{2}\frac{y_{n}+y_{n-1}}{2}
\end{pmatrix}=\boldsymbol 0.
 \label{lorenzDisc}
\end{equation}}Multiplying \eqref{lorenzDisc} by the discrete multiplier $\Lambda^{\tau} = \left(\widetilde{\Lambda}^{\tau} \hspace{2mm} \Sigma^{\tau}\right)$ defined by \eqref{lorenzDiscMult1}, \eqref{lorenzDiscMult2} shows the discrete densities \eqref{lorenzPsi} are preserved.

\subsection{Inviscid Burgers' equation}
Next we illustrate this method for the inviscid Burgers' equation,
\rev{\begin{equation}
F\cont{u}:=u_t+u u_x=0. \label{burgerEqn}
\end{equation}}Writing \eqref{burgerEqn} in conserved form, it is well-known that it has the conservation law,
\rev{\begin{equation*}
\left.\left(D_t u+D_x \left(\frac{u^2}{2}\right)\right)\right|_{F\cont{u}=0}=\left.\lambda\cont{u} \cdot F\cont{u}\frac{}{}\right|_{F\cont{u}=0}=0,
\end{equation*}}which corresponds to a multiplier of \rev{$\lambda\cont{u}=1$}. Now suppose we discretize $\psi$, $\phi$, $\lambda$ with the following expressions,\rev{\begin{align}\psi^{\tau,h}\discdep{u^{\tau,h}}{n,j} &:= u_{n,j}, \nonumber \\
\phi^{\tau,h}\discdep{u^{\tau,h}}{n,j} &:= \frac{u_{n,j}^2}{2}, \label{burgerPsi} \\
\lambda^{\tau,h}\discdep{u^{\tau,h}}{n,j}&:=1. \nonumber
\end{align}}Then, $D_t^{\tau,h}\psi$, $D_x^{\tau,h}\phi$ defined by \eqref{dtScalar}, \eqref{dxScalar} are,
\rev{\begin{align}D_t^{\tau,h}\psi\discdep{u^{\tau,h}}{n,j}&:
  = \frac{u_{n,j}-u_{n-1,j}}{\tau}, \nonumber \\
  D_x^{\tau,h}\phi\discdep{u^{\tau,h}}{n,j}&:
  = \frac{(u_{n,j})^2-(u_{n,j-1})^2}{2h}. \nonumber
\end{align}}So by \eqref{discScalarF}, we obtain the well-known conservative discretization of \eqref{burgerEqn} which preserves \eqref{burgerPsi},
\rev{\begin{equation}
F^{\tau,h}\discdep{u^{\tau,h}}{n,j}:=\frac{u_{n,j}-u_{n-1,j}}{\tau}+\frac{(u_{n,j})^2-(u_{n,j-1})^2}{2h}=0, \label{burgerDisc1}
\end{equation}}{which is clearly first order in time and in space.}

To make this example more interesting, let us consider other conservation laws of \eqref{burgerEqn}. By the method of Euler operator (or inspired guess), the inviscid Burgers' equation has a family of conservation law multiplier \rev{$\lambda\cont{u}=u^{p-1}$} for $p\in \mathbb{N}$ with the conservation law,
\rev{\begin{equation*}
\left.\left(D_t\left(\frac{u^{p}}{p}\right)+D_x \left(\frac{u^{p+1}}{p+1}\right)\right)\right|_{F\cont{u}=0}=\left.\lambda\cont{u}\cdot F\cont{u}\frac{}{}\right|_{F\cont{u}=0}=0.
\end{equation*}}Similar to before, we discretize $\psi$, $\phi$ , $\lambda$ with the expressions,
\rev{\begin{align}
\psi^{\tau,h}\discdep{u^{\tau,h}}{n,j} &:= \frac{(u_{n,j})^p}{p}, \label{burgerPPsi} \\
\phi^{\tau,h}\discdep{u^{\tau,h}}{n,j} &:= \frac{u_{n,j}^{p+1}}{p+1}, \label{burgerPPhi} \\
\lambda^{\tau,h}\discdep{u^{\tau,h}}{n,j}&:=\frac{1}{p}\sum_{k=0}^{p-1}(u_{n,j})^{p-1-k}(u_{n-1,j})^{k}. \label{burgerPDiscMult}
\end{align}}It turns out the choice for the time averaged expression \eqref{burgerPDiscMult} of $\lambda^{\tau,h}_{n,j}$ simplifies the final expression of the discretization. Similar as before, $D_t^{\tau,h}\psi$, $D_x^{\tau,h}\phi$ defined by \eqref{dtScalar}, \eqref{dxScalar} are,
\rev{\begin{subequations}\begin{align}D_t^{\tau,h}\psi\discdep{u^{\tau,h}}{n,j}&:
  = \frac{(u_{n,j})^p-(u_{n-1,j})^p}{p\tau}, \label{burgerDPsi} \\
  D_x^{\tau,h}\phi\discdep{u^{\tau,h}}{n,j}&:
  = \frac{(u_{n,j})^{p+1}-(u_{n,j-1})^{p+1}}{(p+1)h}, \label{burgerDPhi}
\end{align}\end{subequations}}\noindent It can be shown that $\lambda^{\tau,h}$ in \eqref{burgerPPsi} and $D_t^{\tau,h}\psi$ in \eqref{burgerDPsi} are first order accurate in time and $D_x^{\tau,h}\phi$ is first order accurate in space. By Theorem \ref{scalarThm}, the discretization of \eqref{discScalarF} is first order in space and time. Employing the identity,  $\displaystyle a^{p+1}-b^{p+1}=(a-b)\sum_{k=0}^{p}a^{p-k}b^k$, $F^{\tau,h}$ simplifies to,
\rev{
\begin{align}
F^{\tau,h}\discdep{u^{\tau,h}}{n,j}&:=\frac{u_{n,j}-u_{n-1,j}}{\tau}\nonumber\\
&\hspace{6mm}+\frac{p}{p+1}\frac{(u_{n,j})^p+(u_{n,j})^{p-1}u_{n,j-1}+\cdots+u_{n,j}(u_{n,j-1})^{p-1}+(u_{n,j-1})^p}{(u_{n,j})^{p-1}+(u_{n,j})^{p-2}u_{n-1,j}+\cdots+u_{n,j}(u_{n-1,j})^{p-2}+(u_{n-1,j})^{p-1}}\frac{u_{n,j}-u_{n,j-1}}{h}=0, \label{burgerDisc2}
\end{align}}which can readily be checked to preserve the discrete density \eqref{burgerPPsi} \aw{when the fluxes \eqref{burgerPPhi} vanish at the boundaries.} Moreover, \eqref{burgerDisc2} can be seen as a generalization to \eqref{burgerDisc1}, with $p=1$.

\subsection{Korteweg-de Vries equation}
\label{sec:kdv}
Next we consider the Korteweg-de Vries equation,
\rev{\begin{equation}
F\cont{u}:=u_t+uu_x+u_{xxx}=0. \label{kdvEqn}
\end{equation}}As written, \eqref{kdvEqn} does not admit a variational principle. Indeed, it is well known that upon the change of variable $u=v_x$, \eqref{kdvEqn} is the Euler-Lagrange equations of a Lagrangian. Our main motivation here is to show that we can construct a conservative discretization of the Korteweg-de Vries equation without such transformation.

It can be checked (or by the method of Euler operator) that \eqref{kdvEqn} has a multiplier \rev{$\lambda \cont{u}= u$} with the corresponding density and flux,
\rev{\begin{equation*}
\left.\left(D_t \left(\frac{u^2}{2}\right)+D_x \left(\frac{u^3}{3}+uu_{xx}-\frac{u_x^2}{2}\right)\right)\right|_{F\cont{u}=0}=\left.\lambda\cont{u}\cdot F\cont{u}\frac{}{}\right|_{F\cont{u}=0}=0.
\end{equation*}}\noindent To simplify the final expression for $F^{\tau,h}$, we choose to discretize $\psi, \phi, \lambda$ as follows,
\rev{\begin{align}
\psi^{\tau,h}\discdep{u^{\tau,h}}{n,j} &:= \frac{u_{n,j}^2}{2}, \label{kdvPsi} \\
\phi^{\tau,h}\discdep{u^{\tau,h}}{n,j} &:= \frac{u_{n,j}^3}{3}+\left(\frac{u_{n,j+1}+u_{n,j}}{2}\right)\left(\frac{u_{n,j+1}-2u_{n,j}+u_{n,j-1}}{h^2}\right)-\frac{1}{2}\left(\frac{u_{n,j+1}-u_{n,j}}{h}\right)^2, \label{kdvPhi} \\
\lambda^{\tau,h}\discdep{u^{\tau,h}}{n,j}&:= \frac{u_{n,j}+u_{n-1,j}}{2}. \label{kdvMult}
\end{align}}\noindent Thus, \eqref{dtScalar} and \eqref{dxScalar} simplify to
\rev{\begin{align}D_t^{\tau,h}\psi\discdep{u^{\tau,h}}{n,j}&:
  = \frac{(u_{n,j})^2-(u_{n-1,j})^2}{2\tau}, \label{kdvDPsi} \\
  D_x^{\tau,h}\phi\discdep{u^{\tau,h}}{n,j}&:
  = \frac{(u_{n,j})^{3}-(u_{n,j-1})^{3}}{3h}+\left(\frac{u_{n,j}+u_{n,j-1}}{2}\right)\left(\frac{u_{n,j+1}-3u_{n,j}+3u_{n,j-1}-u_{n,j-2}}{h^3}\right). \label{kdvDPhi}
\end{align}}\noindent One can see that $\lambda^{\tau,h}$ in \eqref{kdvMult} and $D_t^{\tau,h}\psi$ in \eqref{kdvDPsi} are first order accurate in time and $D_x^{\tau,h}\phi$ in \eqref{kdvDPhi} is first order in space. Thus, it follows from Theorem \ref{scalarThm} that,
\rev{\begin{align*}
F^{\tau,h}\discdep{u^{\tau,h}}{n,j}:= \frac{u_{n,j}-u_{n-1,j}}{\tau}&+\frac{2(u_{n,j}^2+u_{n,j}u_{n,j-1}+u_{n,j-1}^2)}{3(u_{n,j}+u_{n-1,j})}\left(\frac{u_{n,j}-u_{n,j-1}}{h}\right)\\
&+\left(\frac{u_{n,j}+u_{n,j-1}}{u_{n,j}+u_{n-1,j}}\right)\left(\frac{u_{n,j+1}-3u_{n,j}+3u_{n,j-1}-u_{n,j-2}}{h^3}\right)=0,
\end{align*}}\noindent is a {first order} discretization of \eqref{kdvEqn} that preserves the discrete density \eqref{kdvPsi} when the fluxes \eqref{kdvPhi} vanish at the boundary.

\aw{
\subsection{Shallow water PDE system}
Lastly, we consider the two-dimensional shallow-water equations in non-dimensional form
\rev{\begin{align}
 \boldsymbol F\cont{\boldsymbol u}:= \begin{pmatrix} u_t + \mathbf{v}\cdot\nabla u + \eta_x\\
 v_t + \mathbf{v}\cdot\nabla v + \eta_y\\
 \eta_t + \nabla\cdot (\eta \mathbf{v})\end{pmatrix}=\boldsymbol 0, \label{shallowEqn}
\end{align}}where $\mathbf{v}=(u,v)$ is the two-dimensional velocity field and $\eta$ is the displacement of the water surface over a constant reference level. It is well-known that this system admits conservation of momentum and mass. The multiplier form of these conservation laws is
\rev{\begin{align*}
\left.\left(D_t\begin{pmatrix}\eta u\\\eta v\\\eta\end{pmatrix}+D_x\begin{pmatrix}\eta u^2+\frac12\eta^2\\\eta uv\\\eta u\end{pmatrix}+D_y\begin{pmatrix}\eta uv\\\eta v^2+\frac12\eta^2\\\eta v\end{pmatrix}\right)\right|_{\boldsymbol F\cont{\boldsymbol u}=0}&=\left(\begin{array}{ccc}\eta & 0 & u\\ 0 & \eta & v \\ 0 & 0 & 1\end{array}\right)\left.\left(\begin{array}{c}u_t + \mathbf{v}\cdot\nabla u + \eta_x \\ v_t + \mathbf{v}\cdot\nabla v + \eta_y \\ \eta_t + \nabla\cdot (\eta \mathbf{v})\end{array}\right)\right|_{\boldsymbol F\cont{\boldsymbol u}=0}\\&={\boldsymbol 0}.
\end{align*}}\noindent Choosing the discretization for $\boldsymbol \psi$, $\boldsymbol \phi$ and $\Lambda$,
\rev{\begin{align}
\boldsymbol \psi^{\tau, h}\discdep{\boldsymbol u^{\tau,h}}{n,i,j} &= \begin{pmatrix}\eta_{n,i,j}u_{n,i,j} \\ \eta_{n,i,j}v_{n,i,j} \\ \eta_{n,i,j}\end{pmatrix}, \nonumber\\
\boldsymbol \Phi^{\tau, h}\discdep{\boldsymbol u^{\tau,h}}{n,i,j} &= \begin{pmatrix}\eta_{n,i,j}u_{n,i,j}^2+\frac{1}{2}\eta_{n,i,j}^2 & \eta_{n,i,j}u_{n,i,j}v_{n,i,j} \\ \eta_{n,i,j}u_{n,i,j}v_{n,i,j}& \eta_{n,i,j}v_{n,i,j}^2+\frac{1}{2}\eta_{n,i,j}^2 \\
\eta_{n,i,j}u_{n,i,j}& \eta_{n,i,j}v_{n,i,j}
\end{pmatrix},\nonumber\\
\Lambda^{\tau, h}\discdep{\boldsymbol u^{\tau,h}}{n,i,j} &= \begin{pmatrix} \frac{\eta_{n,i,j}+\eta_{n-1,i,j}}{2} & 0 & \frac{u_{n,i,j}+u_{n-1,i,j}}{2} \\ 0 & \frac{\eta_{n,i,j}+\eta_{n-1,i,j}}{2} & \frac{v_{n,i,j}+v_{n-1,i,j}}{2} \\ 0 & 0 & 1 \end{pmatrix}, \label{shallowMult}
\end{align}}\noindent then $D_t^{\tau,h} \boldsymbol \psi$ and $D_t^{\tau,h} \boldsymbol \phi$ defined by \eqref{dtAugSystem} and \eqref{dxAugSystem} are
\rev{\begin{align}D_t^{\tau,h}\boldsymbol \psi\discdep{\boldsymbol u^{\tau,h}}{n,i,j}&:
  =\begin{pmatrix}
  \frac{\eta_{n,i,j}u_{n,i,j}-\eta_{n-1,i,j}u_{n-1,i,j}}{\tau} \\ \frac{\eta_{n,i,j}v_{n,i,j} -\eta_{n-1,i,j}v_{n-1,i,j}}{\tau} \\ \frac{\eta_{n,i,j}-\eta_{n-1,i,j}}{\tau}
  \end{pmatrix}, \label{shallowDPsi} \\
  D_{\boldsymbol x}^{\tau,h}\cdot\boldsymbol \Phi\discdep{\boldsymbol u^{\tau,h}}{n,i,j}&:
  =\begin{pmatrix}\frac{\eta_{n,i,j}u_{n,i,j}^2- \eta_{n,i-1,j}u_{n,i-1,j}^2}{h}+\frac{1}{2}\frac{\eta_{n,i,j}^2-\eta_{n,i-1,j}^2}{h} + \frac{\eta_{n,i,j}u_{n,i,j}v_{n,i,j}-\eta_{n,i,j-1}u_{n,i,j-1}v_{n,i,j-1}}{h} \\ \frac{\eta_{n,i,j}u_{n,i,j}v_{n,i,j}-\eta_{n,i+1,j}u_{n,i+1,j}v_{n,i+1,j}}{h} + \frac{\eta_{n,i,j}v_{n,i,j}^2- \eta_{n,i,j-1}v_{n,i,j-1}^2}{h}+\frac{1}{2}\frac{\eta_{n,i,j}^2-\eta_{n,i,j-1}^2}{h} \\
\frac{\eta_{n,i,j}u_{n,i,j}-\eta_{n,i-1,j}u_{n,i-1,j}}{h} + \frac{\eta_{n,i,j}v_{n,i,j}-\eta_{n,i,j-1}v_{n,i,j-1}}{h} \end{pmatrix}. \label{shallowDPhi}
\end{align}}
Using \eqref{discSystemF}, the entries of the discretization ${\boldsymbol F}^{\tau,h}$ simplifies to,
\rev{\begin{align}
&\left(\boldsymbol F^{\tau,h}\discdep{\boldsymbol u^{\tau,h}}{n,i,j}\right)_1 \nonumber\\
&\qquad= \frac{u_{n,i,j}-u_{n-1,i,j}}{\tau}+\left(\frac{2\eta_{n,i-1,j}}{\eta_{n,i,j}+\eta_{n-1,i,j}}\right) \left(u_{n,i,j}+u_{n,i-1,j}-\frac{u_{n,i,j}+u_{n-1,i,j}}{2}\right)\left(\frac{u_{n,i,j}-u_{n,i-1,j}}{h}\right) \nonumber\\
&\qquad+\left(\frac{2\eta_{n,i,j}}{\eta_{n,i,j}+\eta_{n-1,i,j}}\right)v_{n,i,j}\left(\frac{u_{n,i,j}-u_{n,i,j-1}}{h}\right) + \left(\frac{\eta_{n,i,j}+\eta_{n,i-1,j}}{\eta_{n,i,j}+\eta_{n-1,i,j}}\right)\left(\frac{\eta_{n,i,j}-\eta_{n,i-1,j}}{h}\right)\nonumber\\
&\qquad+\left(\frac{2}{\eta_{n,i,j}+\eta_{n-1,i,j}}\right)\left(\frac{\eta_{n,i,j}-\eta_{n,i-1,j}}{h}\right)u_{n,i,j}\left(u_{n,i,j}-\frac{u_{n,i,j}+u_{n-1,i,j}}{2}\right)\nonumber \\
&\qquad\left(\frac{2}{\eta_{n,i,j}+\eta_{n-1,i,j}}\right)\left(\frac{\eta_{n,i,j} v_{n,i,j}-\eta_{n,i,j-1}v_{n,i,j-1}}{h}\right)\left(u_{n,i,j-1}-\frac{u_{n,i,j}+u_{n-1,i,j}}{2}\right), \label{shallowDisc1} \\
&\left(\boldsymbol F^{\tau,h}\discdep{\boldsymbol u^{\tau,h}}{n,i,j}\right)_2 \nonumber
\\&\qquad= \frac{v_{n,i,j}-v_{n-1,i,j}}{\tau}+\left(\frac{2\eta_{n,i,j-1}}{\eta_{n,i,j}+\eta_{n-1,i,j}}\right) \left(v_{n,i,j}+v_{n,i,j-1}-\frac{v_{n,i,j}+v_{n-1,i,j}}{2}\right)\left(\frac{v_{n,i,j}-v_{n,i,j-1}}{h}\right)  \nonumber\\
&\qquad+\left(\frac{2\eta_{n,i,j}}{\eta_{n,i,j}+\eta_{n-1,i,j}}\right)u_{n,i,j}\left(\frac{v_{n,i,j}-v_{n,i-1,j}}{h}\right) + \left(\frac{\eta_{n,i,j}+\eta_{n,i,j-1}}{\eta_{n,i,j}+\eta_{n-1,i,j}}\right)\left(\frac{\eta_{n,i,j}-\eta_{n,i,j-1}}{h}\right)\nonumber\\
&\qquad+\left(\frac{2}{\eta_{n,i,j}+\eta_{n-1,i,j}}\right)\left(\frac{\eta_{n,i,j}-\eta_{n,i,j-1}}{h}\right)v_{n,i,j}\left(v_{n,i,j}-\frac{v_{n,i,j}+v_{n-1,i,j}}{2}\right)\nonumber \\
&\qquad\left(\frac{2}{\eta_{n,i,j}+\eta_{n-1,i,j}}\right)\left(\frac{\eta_{n,i,j} u_{n,i,j}-\eta_{n,i-1,j}u_{n,i-1,j}}{h}\right)\left(v_{n,i-1,j}-\frac{v_{n,i,j}+v_{n-1,i,j}}{2}\right), \label{shallowDisc2} \\
&\left(\boldsymbol F^{\tau,h}\discdep{\boldsymbol u^{\tau,h}}{n,i,j}\right)_3 = \frac{\eta_{n,i,j}-\eta_{n-1,i,j}}{\tau}+\frac{\eta_{n,i,j}u_{n,i,j}-\eta_{n,i-1,j}u_{n,i-1,j}}{h} + \frac{\eta_{n,i,j}v_{n,i,j}-\eta_{n,i,j-1}v_{n,i,j-1}}{h}.
\label{shallowDisc3}
\end{align}}Note that, as $h\rightarrow 0$, the first four terms of \eqref{shallowDisc1} and \eqref{shallowDisc2} approaches the limit $u_t+uu_x+vu_y+\eta_x$ and $v_t+vv_y+uv_x+\eta_y$, respectively. Also, the last two terms in \eqref{shallowDisc1} and \eqref{shallowDisc2} vanish, as expected in order for \eqref{shallowDisc1}-\eqref{shallowDisc3} to be consistent with \eqref{shallowEqn}. {In particular, it can be seen that $\Lambda^{\tau, h}$ in \eqref{shallowMult} and $D_t^{\tau,h}\boldsymbol \psi$ in \eqref{shallowDPsi} are first order in time and $D_{\boldsymbol x}^{\tau,h}\cdot\boldsymbol \Phi$ in \eqref{shallowDPhi} is first order in space. Thus, the proposed discretization ${\boldsymbol F}^{\tau,h}$ is first order accurate by Theorem \ref{systemThm}.}
}
\rev{
\section{Consistency when the inverse of the discrete multiplier is singular}
\label{sec:singularMult}
In Sections \ref{sec:scalar}, \ref{sec:squareSystem} and \ref{sec:rectSystem}, consistencies of the multiplier method were shown provided that the inverse of the discrete multiplier $\Lambda^{\tau, h}$ do not become singular on the mesh points in $\mesh$. In general, this may not be satisfied depending on the nature of the discrete solution $\boldsymbol u^{\tau,h}$. We now provide a partial answer for establishing consistency in the case for scalar ODEs when the inverse of the discrete multiplier become singular. We comment on the generalization to system of ODEs in the conclusion.

\subsection{Zero-compatibility condition and consistency for scalar ODEs}
The main idea is as follows: Rather than arbitrary choices of discrete multiplier $\lambda^{\tau}\disc{u^{\tau}}$ and discrete density $D_t^{\tau}\psi\disc{u^{\tau}}$, the pair should be chosen in a ``compatible" manner which mimics the zero properties of $\lambda\cont{u}$ and $D_t\psi\cont{u}$. In particular, if $\lambda\cont{u}$ is a conservation law multiplier of $F\cont{u}$ with density $\psi\cont{u}$, then for any $u$ such that,
\begin{equation}
\lambda\contdep{u}{t_i}=0 \Rightarrow \lambda\contdep{u}{t_i} F\contdep{u}{t_i}=D_t \psi\contdep{u}{t_i} = 0. \label{scalarCZC}
\end{equation}
The discrete analog of \eqref{scalarCZC} leads to the zero-compatibility condition of the multiplier method for scalar ODEs. For clarity, we shall assume $\mu$-step finite difference approximations for the discrete multiplier $\lambda^{\tau}\discdep{u^\tau}{i}=\lambda^{\tau}(t_i; u_{i+1}, u_{i},\dots, u_{i-\mu+1})$ and $D_t^{\tau}\psi\discdep{u^\tau}{i}=D_t^{\tau}\psi(t_i; u_{i+1}, u_{i},\dots, u_{i-\mu+1})$, though the same principle can be applied to other stencils.
\begin{defn} For fixed $t_i$ and $u_{i},\dots, u_{i-\mu+1}$, the pair $(\lambda^{\tau},D_t^{\tau}\psi)$ is called zero-compatible of order $l$ if for any $u_{i+1}$ and $0\leq j \leq l-1$,
\begin{equation*}
\frac{\partial^j \lambda^{\tau}}{(\partial u_{i+1})^j} \discdep{u^\tau}{i}=0 \Rightarrow \frac{\partial^j D_t^{\tau}\psi}{(\partial u_{i+1})^j}\discdep{u^\tau}{i} = 0.
\end{equation*}
\end{defn}\noindent In other words, the zero set of $\frac{\partial^j \lambda^{\tau}}{(\partial u_{i+1})^j}\discdep{u^\tau}{i}$ is a subset of the zero set of $\frac{\partial^j D_t^{\tau}\psi}{(\partial u_{i+1})^j}\discdep{u^\tau}{i}$ while holding fixed $t_i$ and $u_{i},\dots, u_{i-\mu+1}$. In order to show consistency of the multiplier method with vanishing multiplier, we state the following two lemmas with their proofs presented in Appendix C. First, we will need a consistency result for perturbed  $\lambda, D_t\psi$ and $\lambda^\tau, D_t^\tau\psi$.
\begin{lem} Let $\lambda$ be a $r$-th order conservation law multiplier of a $k$-th order scalar ODE with density function $\psi$. Suppose $\lambda^\tau$ and $D^\tau_t\psi$ are finite difference discretizations of $\lambda$ and $D_t\psi$ with accuracy of order $q$. Let $\epsilon \in \mathbb{R}$ and $u, v\in C^{q+\max(k,r)}(\overline{\mathcal{I}})$, then there exists $C_\lambda, C_t$ depending on $t_i$ and $u,v$ such that,
\begin{subequations}
\begin{align}
\lambda\contdep{u+\epsilon v}{t_i}&=\lambda^\tau\discdep{u+\epsilon v}{i}+C_{\lambda}(\epsilon)\tau^q, \label{lem:lambdaRem} \\
D_t\psi\contdep{u+\epsilon v}{t_i}&=D_t^\tau\psi\discdep{u+\epsilon v}{i}+C_t(\epsilon)\tau^q. \label{lem:DpsiRem}
\end{align}
\end{subequations}
Moreover, denote for all $1\leq j \leq q+\max(k,r)-1$, 
\begin{subequations}
\begin{align}
\lambda_j\contdep{u,v}{t_i} := \left.\frac{1}{j!}\frac{d^j}{d\epsilon^j} \lambda\contdep{u+\epsilon v}{t_i}\right|_{\epsilon = 0}, &\quad \lambda_j^\tau\discdep{u,v}{i} := \left.\frac{1}{j!}\frac{d^j}{d\epsilon^j} \lambda^\tau\discdep{u+\epsilon v}{i}\right|_{\epsilon = 0} \label{lem:lambdaTerms} \\
D_t\psi_j\contdep{u,v}{t_i} := \left.\frac{1}{j!}\frac{d^j}{d\epsilon^j} D_t\psi\contdep{u+\epsilon v}{t_i}\right|_{\epsilon = 0}, &\quad D_t^\tau\psi_j\discdep{u,v}{i} := \left.\frac{1}{j!}\frac{d^j}{d\epsilon^j} D_t^\tau\psi\discdep{u+\epsilon v}{i}\right|_{\epsilon = 0} \label{lem:DpsiTerms}
\end{align}\end{subequations}\noindent If $C_\lambda, C_t$ are $``q+\max(k,r)"$-times continuously differentiable at $\epsilon=0$, then for $1\leq j \leq q+\max(k,r)-1$,
\begin{subequations}
\begin{align}
\lambda_j\contdep{u,v}{t_i} &= \lambda_j^\tau\discdep{u,v}{i}+C_\lambda^{(j)}(0)\tau^q, \label{lem:lambdaiExp}\\
D_t\psi_j\contdep{u,v}{t_i} &= D_t^\tau\psi_j\discdep{u,v}{i}+C_\lambda^{(j)}(0)\tau^q. \label{lem:DpsiiExp}
\end{align}
\end{subequations}
\label{lem:taylorExp}
\end{lem}
\vskip -5mm
\noindent The next lemma is essentially L'H\^{o}pital's rule stated for our purpose for convenience.
\begin{lem} Fix $t_i$ and $u_{i},\dots, u_{i-\mu+1}$ and let the pair $(\lambda^{\tau}, D_t^{\tau}\psi)$ be zero-compatible of order $l$. Suppose $u_{i+1}=z$ is an isolated zero of order $l$ of $\lambda^{\tau}$ and $\left(\frac{d}{du_{i+1}}\right)^{l}\lambda^{\tau}\discdep{u}{i}\neq 0$ at $u_{i+1}=z$, then
\begin{equation*}
\lim_{u_{i+1}\rightarrow z} \frac{D_t^{\tau}\psi\discdep{u^\tau}{i}}{\lambda^{\tau}\discdep{u^\tau}{i}} = \left.\frac{\frac{\partial^j D_t^{\tau}\psi}{(\partial u_{i+1})^j}\discdep{u^\tau}{i}}{\frac{\partial^j \lambda^{\tau}}{(\partial u_{i+1})^j}\discdep{u^\tau}{i}}\right|_{u_{i+1}=z}.
\end{equation*}\label{scalarZClemma}\end{lem}
\vskip -5mm
\noindent Now, we are in the position to show the consistency for the multiplier method when the inverse of the discrete multiplier is singular in the following sense. Note that the zero-compatibility condition will play a vital role in the error estimate.
\begin{thm}
Let $u\in C^{q+\max(k,r)}(\overline{\mathcal{I}})$ and $\lambda$ be a $r$-th order conservation law multiplier of a $k$-th order scalar ODE,
$$
F\contdep{u}{t}=0, \text{ in }t\in\mathcal{I},
$$ with density $\psi$. Suppose $(\lambda^{\tau}, D_t^{\tau}\psi)$ is a zero-compatible discretization of order $l$ with accuracy of order $q$ and $u_{i+1}=z$ is an isolated zero of order $l$ of $\lambda^{\tau}$ for all $\tau\leq \tau_0$. Define the finite difference discretization of $F$,
$$
F^\tau\discdep{u^\tau}{i} = \left\{\begin{split}
\frac{D_t^{\tau}\psi\discdep{u^\tau}{i}}{\lambda^{\tau}\discdep{u^\tau}{i}} &, \text{ if } u_{i+1} \neq z, \\
\frac{\frac{\partial^l D_t^{\tau}\psi}{(\partial u_{i+1})^l}\discdep{u^\tau}{i}}{\frac{\partial^l \lambda^{\tau}}{(\partial u_{i+1})^l}\discdep{u^\tau}{i}} &, \text{ if } u_{i+1}= z.
\end{split}\right.
$$
Then for $u(t_{i+1})=z$, there exists a constant $C>0$ independent of $\tau$ such that for all $\tau\leq \tau_0$,
$$
|F\contdep{u}{t_i} - F^\tau\discdep{u}{i}|  \leq C\tau^q.
$$ \label{scalarZCThm}
\end{thm}
\begin{proof}
Fix any $\tau\leq \tau_0$ and let $v(t)$ be a bump function which is compactly supported on $[t_{i+\frac{1}{2}},t_{i+\frac{3}{2}}]$ with $v(t_{i+1})=1$, i.e. $v\in C^\infty_0([t_{i+\frac{1}{2}},t_{i+\frac{3}{2}}])$. For $\epsilon$ to be chosen,
\begin{align*}
|F\contdep{u}{t_i}-F^\tau\discdep{u}{t_i}| &= \left|F\contdep{u}{t_i}-\frac{\frac{\partial^l D_t^{\tau}\psi}{(\partial u_{i+1})^l}\discdep{u^\tau}{i}}{\frac{\partial^l \lambda^{\tau}}{(\partial u_{i+1})^l}\discdep{u^\tau}{i}}\right| \leq E_1 + E_2 + E_3, \hspace{.8cm} \text{ where}, \\
E_1 &:= \left|F\contdep{u}{t_i} - F\contdep{u+\epsilon v}{t_i}\right|, \\
E_2 &:= \left|F\contdep{u+\epsilon v}{t_i} - \frac{D_t^{\tau}\psi\discdep{u+\epsilon v}{i}}{\lambda^{\tau}\discdep{u+\epsilon v}{i}}\right|, \\
E_3 &:= \left|\frac{D_t^{\tau}\psi\discdep{u+\epsilon v}{i}}{\lambda^{\tau}\discdep{u+\epsilon v}{i}}-\frac{\frac{\partial^l D_t^{\tau}\psi}{(\partial u_{i+1})^l}\discdep{u^\tau}{i}}{\frac{\partial^l \lambda^{\tau}}{(\partial u_{i+1})^l}\discdep{u^\tau}{i}} \right|.
\end{align*} 
We proceed to bound each $E_i$ as follows.

For $E_1$, set $f(\epsilon):=F\contdep{u+\epsilon v}{t_i}$ and fix a $\delta_1>0$. Since $F$ is continuously differentiable in its arguments, by the mean value theorem for $f$, there exists $\xi$ between $0$ and $\epsilon$ such that for all $|\epsilon| \leq \delta_1$,
\begin{align}
E_1&=\left|f(\epsilon)-f(0)\right| = |f'(\xi)| |\epsilon| = |\epsilon| \left|\sum_{j=0}^k \frac{\partial F}{\partial u^{(j)}}\contdep{u+\xi v}{t_i} v^{(j)}(t_i)\right| \leq |\epsilon| \underbrace{\max_{\substack{\xi\in [-\delta_1,\delta_1] \\ t\in \overline{\mathcal{I}}}} \left|\sum_{j=0}^k \frac{\partial F}{\partial u^{(j)}}\contdep{u+\xi v}{t} v^{(j)}(t)\right|}_{=: M_1}, \label{ZCthm_e1}
\end{align} where the constant $M_1$ exists by our regularity assumption on $F, u ,v$ and that $[-\delta_1,\delta_1]$ and $\overline{\mathcal{I}}$ are compact.

For $E_2$, we can proceed in a similar manner as the non-singular multiplier case as long as we can pick $\epsilon$ so that $\lambda^\tau\discdep{u+\epsilon v}{i}\neq 0$. In particular, we would have
\begin{align}
E_2 &\leq \frac{1}{|\lambda^\tau\discdep{u+\epsilon v}{i}|}\left(\frac{}{}|\lambda^\tau\discdep{u+\epsilon v}{i}-\lambda\contdep{u+\epsilon v}{t_i}||F\contdep{u+\epsilon v}{t_i}|+|D_t\psi\contdep{u+\epsilon v}{t_i}-D_t^{\tau}\psi\discdep{u+\epsilon v}{i}|\right) \label{ZCthm_e2}
\end{align}
To bound each term in \eqref{ZCthm_e2}, first recall that $v(t_j)=0$ for all $j\neq i+1$ and $v(t_{i+1})=1$. So from \eqref{lem:lambdaTerms}, for any $1\leq j\leq l$,
\begin{align}
\lambda^\tau_j\discdep{u,v}{i} &:= \left.\frac{1}{j!}\frac{d^j}{d\epsilon^j} \lambda^\tau\discdep{u+\epsilon v}{i}\right|_{\epsilon = 0} \nonumber \\
&= \sum_{\substack{|\alpha|=j\\ \alpha=(\alpha_0, \dots, \alpha_\mu)}} \frac{1}{\alpha!} \frac{\partial^j \lambda^\tau\discdep{u}{i}}{(\partial u_{i+1})^{\alpha_0}\cdots (\partial u_{i-\mu+1})^{\alpha_\mu}} v(t_{i+1})^{\alpha_0}\cdots v(t_{i-\mu+1})^{\alpha_\mu} \nonumber\\
&= \frac{1}{j!} \frac{\partial^j \lambda^{\tau}}{(\partial u_{i+1})^j}\discdep{u}{i}. \label{ZCS_e1}
\end{align}
Since $u_{i+1}=z$ is a zero of order $l$ of $\lambda^\tau$, \eqref{ZCS_e1} implies for all $\tau\leq \tau_0$,
\begin{equation}
\lambda^\tau_j\discdep{u,v}{i} =\left \{\begin{split} &0,& 1\leq j\leq l-1,\\
&\frac{1}{l!} \frac{\partial^l \lambda^{\tau}}{(\partial u_{i+1})^l}\discdep{u}{i}\neq 0,& j=l.
\end{split}
\right. \label{ZCS_e2}
\end{equation}Similarly, from \eqref{lem:DpsiTerms}, for any $1\leq j\leq l$,
\begin{align}
D_t^\tau\psi_j\discdep{u,v}{i} &:= \left.\frac{1}{j!}\frac{d^j}{d\epsilon^j} D_t^\tau\psi\discdep{u+\epsilon v}{i}\right|_{\epsilon = 0} \nonumber \\
&= \sum_{\substack{|\alpha|=j\\ \alpha=(\alpha_0, \dots, \alpha_\mu)}} \frac{1}{\alpha!} \frac{\partial^j D_t^\tau\psi\discdep{u}{i}}{(\partial u_{i+1})^{\alpha_0}\cdots (\partial u_{i-\mu+1})^{\alpha_\mu}} v(t_{i+1})^{\alpha_0}\cdots v(t_{i-\mu+1})^{\alpha_\mu} \nonumber\\
&= \frac{1}{j!} \frac{\partial^j D_t^\tau\psi}{(\partial u_{i+1})^j}\discdep{u}{i}. \label{ZCS_e3}
\end{align}Since $(\lambda^{\tau}, D_t^{\tau}\psi)$ is zero-compatible of order $l$ and $u_{i+1}=z$ is a zero of order $l$ of $\lambda^{\tau}$, \eqref{ZCS_e3} implies for $\tau\leq \tau_0$,
\begin{equation}
D_t^\tau\psi_j\discdep{u,v}{i} = \left \{\begin{split} &0,& 1\leq j\leq l-1,\\
&\frac{1}{l!} \frac{\partial^l D_t^{\tau}\psi}{(\partial u_{i+1})^l}\discdep{u}{i},& j=l.
\end{split}
\right. \label{ZCS_e4}
\end{equation}
Moreover, since $v(t)$ is compactly supported on $[t_{i+\frac{1}{2}},t_{i+\frac{3}{2}}]$, for $1\leq j\leq l-1$ and $n=\max(r,k)$,
\begin{subequations}
\begin{align}
\lambda_j\contdep{u,v}{t_i}:= \left.\frac{1}{j!}\frac{d^j}{d\epsilon^j} \lambda\contdep{u+\epsilon v}{t_i}\right|_{\epsilon = 0} &= \sum_{\substack{|\alpha|=j\\ \alpha=(\alpha_0, \dots, \alpha_r)}} \frac{1}{\alpha!} \frac{\partial^j \lambda\contdep{u}{t_i}}{(\partial u)^{\alpha_0}\cdots (\partial u^{(r)})^{\alpha_r}} v(t_i)^{\alpha_0}\cdots {v^{(r)}(t_i)}^{\alpha_r}=0, \label{ZCS_e5}\\
D_t\psi_j\contdep{u,v}{t_i}:= \left.\frac{1}{j!}\frac{d^j}{d\epsilon^j} D_t\psi\contdep{u+\epsilon v}{t_i}\right|_{\epsilon = 0} &= \sum_{\substack{|\alpha|=j\\ \alpha=(\alpha_0, \dots, \alpha_{n-1})}} \frac{1}{\alpha!} \frac{\partial^j D_t\psi\contdep{u}{t_i}}{(\partial u)^{\alpha_0}\cdots (\partial u^{(n-1)})^{\alpha_{n-1}}} v(t_i)^{\alpha_0}\cdots v^{(n-1)}(t_i)^{\alpha_{n-1}}\nonumber\\&=0. \label{ZCS_e6}
\end{align}
\end{subequations}Also, by hypothesis $\left.\lambda^\tau\discdep{u}{i}\right|_{u(t_{i+1})=z}=0$ for all $\tau\leq \tau_0$. So by consistency of $\lambda^\tau$ and $D_t^\tau\psi$, as $\tau\rightarrow 0$,
\begin{subequations}
\begin{align}
|\lambda\contdep{u}{t_i}| &\leq \underbrace{|\lambda^\tau\discdep{u}{i}|}_{=0}+\mathcal{O}(\tau^q) \Rightarrow \lambda\contdep{u}{t_i}=0, \label{ZCS_e6a} \\
|D_t\psi\contdep{u}{t_i}| &\leq \underbrace{|D_t^\tau\psi\discdep{u}{i}|}_{=0}+\mathcal{O}(\tau^q) \Rightarrow D_t\psi\contdep{u}{t_i}=0. \label{ZCS_e6b}
\end{align}
\end{subequations}
Thus, combining \eqref{ZCS_e2}, \eqref{ZCS_e5} and \eqref{ZCS_e6a} with Lemma \ref{lem:taylorExp} yields for some $C_1\geq 0$ independent of $\epsilon$ and $\tau$,
\begin{align}
&|\lambda^\tau\discdep{u+\epsilon v}{i}-\lambda\contdep{u+\epsilon v}{t_i}| \nonumber\\
&\leq \underbrace{|\lambda^\tau\discdep{u}{i}-\lambda\contdep{u}{t_i}|}_{=0}+\underbrace{|\lambda^\tau_1\discdep{u,v}{i}-\lambda_1\contdep{u,v}{t_i}|}_{=0}|\epsilon|+\dots + |\underbrace{\lambda^\tau_{l}\discdep{u,v}{i}-\lambda_l\contdep{u,v}{t_i}}_{=C_\lambda^{(l)}(0) \tau^q}||\epsilon|^{l}+ \mathcal{O}(\tau^q |\epsilon|^{l+1})\nonumber\\ 
& = \tau^q|\epsilon|^l(C_1+\mathcal{O}(|\epsilon|)).\label{ZCS_e7}
\end{align} Similarly, combining \eqref{ZCS_e4}, \eqref{ZCS_e6}  and \eqref{ZCS_e6b} with Lemma \ref{lem:taylorExp} shows for some $C_2\geq 0$ independent of $\epsilon$ and $\tau$,
\begin{align}
|D_t^\tau\psi\discdep{u+\epsilon v}{i}-D_t\psi\contdep{u+\epsilon v}{t_i}| \leq \tau^q|\epsilon|^l(C_2+\mathcal{O}(|\epsilon|)).\label{ZCS_e8}
\end{align}
Finally, for some $\delta_2>0$, $\lambda^\tau\discdep{u+\epsilon v}{i}\neq 0$ for all $|\epsilon|<\delta_2$. Indeed, since from \eqref{ZCS_e2}, for 
$C_3=\frac{1}{l!} \frac{\partial^l \lambda^{\tau}}{(\partial u_{i+1})^l}\discdep{u}{i} \neq 0$,
\begin{align}
\lambda^\tau\discdep{u+\epsilon v}{i} &= \underbrace{\lambda^\tau\discdep{u}{i}}_{=0}+ \underbrace{\lambda_1^\tau\discdep{u,v}{i}}_{=0} \epsilon + \dots + \underbrace{\lambda_l^\tau\discdep{u,v}{i}}_{\neq 0} \epsilon^l+\mathcal{O}(\epsilon^{l+1}) = \epsilon^l (C_3 +\mathcal{O}(\epsilon )) \neq 0 \text{ for }|\epsilon|<\delta_2.\label{ZCS_e9}
\end{align}Let $\displaystyle M_2:=\max_{\epsilon\in [-\delta_2,\delta_2]} |F\contdep{u+\epsilon v}{t_i}|$. Hence, together with \eqref{ZCS_e7}-\eqref{ZCS_e9}, $E_2$ from \eqref{ZCthm_e2} can be bounded as,
\begin{align}
E_2 &\leq \frac{\tau^q |\epsilon|^l (C_1+\mathcal{O}(|\epsilon|))M_2+\tau^q|\epsilon|^l (C_2+\mathcal{O}(|\epsilon|))}{|\epsilon|^l|C_3+\mathcal{O}(\epsilon)|} \nonumber\\
&= \tau^q\frac{C_1M_2+C_2 + \mathcal{O}(|\epsilon|)}{|C_3+\mathcal{O}(\epsilon)|}\nonumber\\
&\leq\tau^q\mathcal{O}(1+|\epsilon|), \text{ if }|\epsilon|<\delta_3\text{ for some }\delta_3>0. \label{ZCthm_e3}
\end{align}

Lastly for $E_3$, since $v(t)$ is compactly supported inside $[t_{i+\frac{1}{2}},t_{i+\frac{3}{2}}]$ and $v(t_{i+1})=1$, then Lemma \ref{scalarZClemma} implies that there exists $\delta_4>0$ such that if $|\epsilon|<\delta_4$,
\begin{align}
E_3 &= \left|\frac{D_t^{\tau}\psi(t_i;z+\epsilon,u(t_i),u(t_{i-1}),\dots,u(t_{i-\mu+1}))}{\lambda^{\tau}(t_i;z+\epsilon,u(t_i),u(t_{i-1}),\dots,u(t_{i-\mu+1}))} - \frac{\frac{\partial^l D_t^{\tau}\psi}{(\partial u_{i+1})^l}\discdep{u^\tau}{i}}{\frac{\partial^l \lambda^{\tau}}{(\partial u_{i+1})^l}\discdep{u^\tau}{i}} \right| \leq \tau^q \label{ZCthm_e4}
\end{align} Combining \eqref{ZCthm_e1}, \eqref{ZCthm_e3}, \eqref{ZCthm_e4} and choosing $\epsilon$ such that $0<|\epsilon|<\min(\delta_1, \delta_2, \delta_3, \delta_4, \tau^q)$ implies,
$$|F\contdep{u}{t_i}-F^\tau\discdep{u}{t_i}| \leq \underbrace{E_1}_{\leq M_1 \tau^q} + \underbrace{E_2}_{\leq C\tau^q(1+\tau^q)} + \underbrace{E_3}_{\leq \tau^q}= \mathcal{O}(\tau^q).$$\end{proof}\subsection{Zero-compatible ODE examples}
In practice, zero-compatible pair of $(\lambda^{\tau},D_t^{\tau}\psi)$ often arises when $D^\tau\psi$ factors algebraically into a product of $\lambda^\tau$ and some closed form expression $G^\tau$, i.e. $D^\tau\psi\disc{u^\tau}=\lambda^\tau\disc{u^\tau}G^\tau\disc{u^\tau}$. In particular, the singularities arising from the zeros of $\lambda^\tau$ become ``removable" since,
$$F^\tau\disc{u^\tau} := (\lambda^\tau)^{-1}\disc{u^\tau} D^\tau\psi\disc{u^\tau} = G^\tau\disc{u^\tau}.$$ Moreover, since $(\lambda^\tau,\lambda^\tau G^\tau)$ is automatically zero compatible of the same order, the consistency result of Theorem \ref{scalarZCThm} implies,
$$|F\contdep{u}{t_i}-G^\tau\discdep{u}{i}| = \mathcal{O}(\tau^q).$$
To illustrate this idea on a concrete example, consider the harmonic oscillator with mass $m$ and spring constant $k$,
\begin{align*}
&F\cont{x}:=mx_{tt}+kx, \text{ where } \left.D_t\left(\frac{m}{2}x_t^2+\frac{k}{2}x^2\right)\right|_{F\cont{x}=0}=\lambda\cont{x} F\cont{x} = 0,
\end{align*}with the conservation law multiplier $\lambda\cont{x}=x_t$.\\
\noindent Choosing $\phi^\tau\discdep{x^\tau}{n}=\frac{m}{2}\left(\frac{x_{n+1}-x_n}{\tau}\right)^2+\frac{k}{2}\left(\frac{x_{n+1}+x_n}{2}\right)^2$, then $D_t^\tau\phi\disc{x^\tau}_{n}$ factors into,
\[
D_t^\tau\phi\discdep{x^\tau}{n} = \frac{\phi^\tau\disc{x^\tau}_{n}-\phi^\tau\disc{x^\tau}_{n-1}}{\tau} = \left(\frac{x_{n+1}-x_{n-1}}{2\tau}\right)\underbrace{\left(m\frac{x_{n+1}-2x_n+x_{n-1}}{\tau^2}+k\frac{x_{n+1}+2x_n+x_{n-1}}{4}\right)}_{G^\tau\discdep{x}{n}}.
\] Indeed, if we now choose $\lambda^\tau\discdep{x^\tau}{n} = \frac{x_{n+1}-x_{n-1}}{2\tau}$, then the multiplier method exactly factors into the following conservative scheme,
\[
F^\tau\discdep{x^\tau}{n} = m\frac{x_{n+1}-2x_n+x_{n-1}}{\tau^2}+k\frac{x_{n+1}+2x_n+x_{n-1}}{4}.
\]While it can already be seen explicitly that $F^\tau\discdep{x^\tau}{n}$ is consistent to $F\contdep{x}{t_n}$, Theorem \ref{scalarZCThm} also guarantees consistency, since $x_{n+1}=x_{n-1}$ is a zero of order one of $\lambda^\tau\discdep{x^\tau}{n}$ with $\frac{\partial\lambda^\tau}{\partial x_{n+1}}\discdep{x^\tau}{n}=\frac{1}{2\tau}\neq 0$.

An example for which $D^\tau\psi\disc{u^\tau}$ does not factors algebraically but Theorem \ref{scalarZCThm} would still apply is the pendulum problem. Recall from Section \ref{sec:pendProblem} that the multiplier method produced,
\begin{align}
D_t^{\tau}\psi\discdep{\theta^{\tau}}{n}& = \frac{(\theta_{n+1}-\theta_{n-1})(\theta_{n+1}-2\theta_n+\theta_{n-1})}{2\tau^3}-\frac{g}{l}\frac{\cos(\theta_{n+1})-\cos(\theta_{n-1})}{2\tau},
\end{align}with the discrete multiplier $\lambda^{\tau}\discdep{\theta^{\tau}}{n}=\frac{\theta_{n+1}-\theta_{n-1}}{2\tau}$. Clearly, $D_t^{\tau}\psi$ cannot be factored algebraically into a product $\lambda^{\tau}$ and a closed form $G^\tau$. However, it can be seen that $(\lambda^{\tau},D_t^{\tau}\psi)$ is zero-compatible of order one and thus Theorem \ref{scalarZCThm} guarantees consistency of $F^\tau$ even when $\lambda^\tau$ vanishes.

}
\section{Conclusion}
\aw{

%One practical way to avoid this issue is to time march the discretized density without multiplying the inverse of discretized multiplier and .

In this paper, we introduced the multiplier method for ordinary and partial differential equations which conserves discretely first integrals and conservation laws. In particular, \rev{when the inverse of the discrete multiplier is non-singular,} we showed that the proposed discretization is consistent to any order of accuracy depending on the choice of finite difference approximation of the conservation law multipliers, densities and fluxes. Also, we showed a discrete version of the divergence theorem holds and thus the densities are conserved exactly at the discrete level. Due to the generality of the consistency theorem of the proposed discretization, we believe there is much advantage and flexibility in achieving higher order accuracy for the densities as well. Moreover, as there is some freedom to choose discretization for the multipliers, densities and fluxes, we can often simplify the final form of the discretization; as demonstrated in the examples of Section \ref{sec:examples}. \rev{Finally, in the case when the inverse of the discrete multiplier becomes singular, we have shown the consistency of the multiplier method for scalar ODEs, provided the pair of discrete multiplier and density is zero-compatible.}

The work presented can be applied to general differential equations possessing conservation laws. Since our approach does not require the existence of any other geometric structure apart from the differential equations and conservation laws themselves, this work is complementary to existing conservative methods. In particular, the multiplier method can be applied even when the underlying differential equations do not admit a Hamiltonian or variational structure, in contrast to geometric integrators and multi-symplectic integrators. Moreover, an important new physical case covered by the multiplier method are dissipative systems; this was illustrated in Section \ref{sec:dho} by the damped harmonic oscillator example. Indeed, it may be possible in some cases to transform a system into one possessing a Hamiltonian or variational structure for which existing methods can be applied. However, depending on applications, it may be more natural and simpler to work with the physical variables themselves. Also, such transformations may come at a loss of accuracy of the discretization, as in the case of the Korteweg-de Vries example in Section \ref{sec:kdv}.

We have investigated the case when the given differential equations possess at most the same number of conservation laws as the number of equations. This requirement was relevant in order to locally invert a full (or sub) multiplier matrix formed by the associated conservation law multipliers. The case not covered here is when there are more conservation laws than the number of equations. The main difficulty stems from the resulting multiplier matrix not being full rank and thus it cannot be treated in the same manner as the other cases. Nonetheless, this is an important case which arises strictly in partial differential equations, where there may be more conservation laws than the number of equations. It is of current interest to extend our results to include that case also.

\rev{In the case when the inverse of the discrete multiplier become singular, it turned out the zero-compatibility condition played a vital role in showing consistency of the multiplier method in scalar ODEs. We believe this can be extended to system of ODEs, where an analogue of \eqref{scalarCZC} is when square multiplier matrices $\Lambda$ becomes singular, i.e. $\det(\Lambda)=0$. In particular, consider a square conservation law multiplier matrix $\Lambda\cont{\boldsymbol u}$ of $\boldsymbol F\cont{\boldsymbol u}$ with density $\boldsymbol\psi\cont{\boldsymbol u}$, then observe for $\boldsymbol u$ such that,
\begin{equation}
\det(\Lambda\contdep{\boldsymbol u}{t_i})=0 \Rightarrow \adjugate(\Lambda\contdep{\boldsymbol u}{t_i})D_t \boldsymbol\psi\contdep{\boldsymbol u}{t_i} = \adjugate(\Lambda\contdep{\boldsymbol u}{t_i}) \Lambda\contdep{\boldsymbol u}{t_i} \boldsymbol F\contdep{\boldsymbol u}{t_i} = \det(\Lambda\contdep{\boldsymbol u}{t_i}) \boldsymbol F\contdep{\boldsymbol u}{t_i} = \boldsymbol 0, \label{systemCZC}
\end{equation}where $\adjugate(\Lambda\contdep{\boldsymbol u}{t_i})$ denotes the adjugate of the multiplier matrix $\Lambda\contdep{\boldsymbol u}{t_i}$.
Thus, we conjecture that the discrete analog of \eqref{systemCZC} will lead to the zero-compatibility condition for system of ODEs and in helping establishing consistency for this case. This is a subject of our current investigation as well.
}

Also, we note that in this paper we did not consider questions pertaining to stability and convergence for the multiplier method.  As the proposed method generally leads to nonlinear discretizations, answering these questions is difficult in general and is the focus of our current research direction. 
}
\aw{
\section*{Acknowledgement}
AB is a recipient of an APART Fellowship of the Austrian Academy of Sciences. JCN acknowledges partial support from the NSERC Discovery and NSERC Discovery Accelerator Supplement programs.

\bibliographystyle{hep}
\bibliography{refs}

\begin{appendices}

\section{Derivation of conservation law for the damped harmonic oscillator}
For illustration, we include a short discussion on how the method of Euler operator can be used to derive conservation laws for the damped harmonic oscillator. For more details and its general theory, see \cite{Olv00, BluAnc10} for many more examples.

Given a scalar function $f(t,x,x_t,x_{tt})$, the Euler operator $\mathcal{E}$ applied to $f$ is given by
\begin{equation*}
\mathcal{E}(f) = \frac{\partial f}{\partial x}-\frac{d}{dt}\left(\frac{\partial f}{\partial x_t}\right)+\frac{d^2}{dt^2}\left(\frac{\partial f}{\partial x_{tt}}\right). \label{euler2ndDer}
\end{equation*}
The special property of $\mathcal{E}$ is that its kernel is exactly characterized by divergence expressions. In particular, for any analytic function $x(t)$,
\begin{equation}
0=\mathcal{E}(f) \iff f(t,x,x_t,x_{tt})=D_t \psi(t,x,x_t) \text{ for some analytic function } \psi(t,x,x_t). \label{eulerProp}
\end{equation}
Thus, to find conservation laws of differential equations of the form $F(t,x,x_t,x_{tt})=0$, we first find conservation law multipliers $\lambda$ such that $\mathcal{E}(\lambda F)=0$. If such a $\lambda$ is found, then there must be $\psi$ such that $\lambda F = D_t \psi$ by \eqref{eulerProp}. So the second step is to compute the density $\psi$ (or fluxes in general) associated with $\lambda$. 

Indeed, if the expression for the differential equation $F$ is already a divergence expression, then the conservation law multiplier $\lambda=1$ would suffice. However, there may be other conservation laws for $F$ which are ``hidden". The goal is uncover other conserved quantities by first finding nontrivial conservation law multipliers $\lambda$. We illustrate this for the damped harmonic oscillator equation $F:=mx_{tt}+\gamma x_t+kx$ by choosing $\lambda$ to have the dependence on $t,x,x_t$. So by \eqref{eulerProp}, after a laborious calculation, 
\begin{align}
0&=\mathcal{E}(\lambda(t,x,x_t) F(t,x,x_t,x_{tt})) = \frac{\partial}{\partial x}\left(\lambda F\right)-\frac{d}{dt}\left(\frac{\partial}{\partial x_t}(\lambda F)\right)+\frac{d^2}{dt^2}\left(\frac{\partial }{\partial x_{tt}}(\lambda F)\right).\nonumber\\
\Rightarrow 0&= x_{tt}\left(\frac{}{}2m\lambda_x -2\gamma\lambda_{x_t}+m\lambda_{x_t t}+mx_t\lambda_{x_t x}-(kx+\gamma x_t)\lambda_{x_t x_t}\right) \nonumber\\
&\hspace{4mm}+x_t\left(\frac{}{}m(2\lambda_{x t}+\lambda_{x_t t t})-2\gamma \lambda_{x_t t} \right)+x\left(\frac{}{}2\lambda_x b-k\lambda_{x_t t}-\gamma\lambda_{x t}+m\lambda_{x t t}\right)\label{appendix:e1}
\end{align} The expression multiplying $x_{tt}$ in \eqref{appendix:e1} does not depend on $x_{tt}$ and hence must vanish independently, as $x(t)$ is arbitrary.
\begin{equation}
2m\lambda_x -2\gamma\lambda_{x_t}+m\lambda_{x_t t}+mx_t\lambda_{x_t x}-(kx+\gamma x_t)\lambda_{x_t x_t}=0 \label{appendix:e2}
\end{equation} To solve \eqref{appendix:e2}, we choose the ansatz of $\lambda(t,x,x_t) = a(t)x_t+b(t)x$ which simplifies to,
\begin{equation}
2mb-2\gamma a+ma'=0. \label{appendix:e3}
\end{equation} In general, there may be other conservation law multipliers of the form $\lambda(t,x,x_t)$ and finding all such multipliers is a classification problem for the given differential equation. Since our main goal here is just to find one conservation law multiplier, the above ansatz allows us to proceed without solving the general problem of \eqref{appendix:e2}.
Similar to \eqref{appendix:e3}, the expressions multiplying $x_t, x$ must also vanish independently and simplify to,
\begin{align}
2mb'-2\gamma a'+ma''&=0, \label{appendix:e4} \\
-a'k+2bk-\gamma b'+mb''&=0. \label{appendix:e5}
\end{align} Equation \eqref{appendix:e4} is just a differential consequence of \eqref{appendix:e3} and hence does not contain new information. To solve \eqref{appendix:e5}, we solve $b$ in terms of $a, a'$ in \eqref{appendix:e3} and substitute it in \eqref{appendix:e5} for $b, b'$ and $b''$. Hence, we find the determining equation for $a(t)$ to be,
\begin{equation}
\frac{m}{2} a''' - \frac{3\gamma}{2}a''+\left(2k+\frac{\gamma^2}{m}\right)a'+\frac{2\gamma k}{m}a=0.\label{appendix:e6}
\end{equation} Using the standard method of characteristic polynomial to solve \eqref{appendix:e6}, we find that $a(t)$ is the linear combination,
$$a(t) = C_1\exp\left(\frac{\gamma}{m}t\right)+C_2\exp\left(\frac{\gamma+\sqrt{\gamma^2-4mk}}{m}t\right)+C_3\exp\left(\frac{\gamma-\sqrt{\gamma^2-4mk}}{m}t\right).$$
For simplicity, we choose $a(t) = e^{\frac{\gamma}{m}t}$ which gives $b(t)=\frac{\gamma}{2m}e^{\frac{\gamma}{m}t}$ by \eqref{appendix:e3}.
Thus, we have found a conservation law multiplier $\lambda(t,x,x_t) = e^{\frac{\gamma}{m}t}\left(x_t+\frac{\gamma}{2m}x\right)$.
To find associated density $\psi$, it is straightforward to match on both sides of,
$$
\psi_t+\psi_x x_t+\psi_{x_t}x_{tt}=D_t\psi = \lambda F = e^{\frac{\gamma}{m}t}\left(x_t+\frac{\gamma}{2m}x\right)\left(mx_{tt}+\gamma x_t+kx\right),
$$
and obtain,
\begin{equation}
\psi(t,x,x_t) = \frac{e^{\frac{\gamma}{m} t}}{2}\left(m\left(x_t+\frac{\gamma}{2m}x\right)^2+\left(k-\frac{\gamma^2}{4m}\right)x^2\right). \label{appendix:e7}
\end{equation}
In summary, using the method of Euler operator, we have derived a conserved quantity $\psi$ given by \eqref{appendix:e7} for the damped harmonic oscillator.  

{
\section{Numerical verification for the damped harmonic oscillator}
In this appendix, we numerically verify our theoretical findings for
the damped harmonic oscillator example in Section \ref{sec:dho}. Specifically, we observed that the proposed scheme has second order accuracy in the variable $x$. Additionally, we check that the density $\psi$ is discretely conserved and that
the discrete density converges in second order to the exact value of $\psi$.

For our numerical experiment, we consider the case where $m=1$, $k=5$,
and $\gamma=1/2$, with initial conditions $x\left(0\right)=1$, and
$x_{t}\left(0\right)=0$. Figure (\ref{NumSolution}) shows the solution
computed with $N=200$ points, which is in good agreement with the
exact solution.

\begin{figure}[h!]
\centering
\includegraphics[width=16cm]{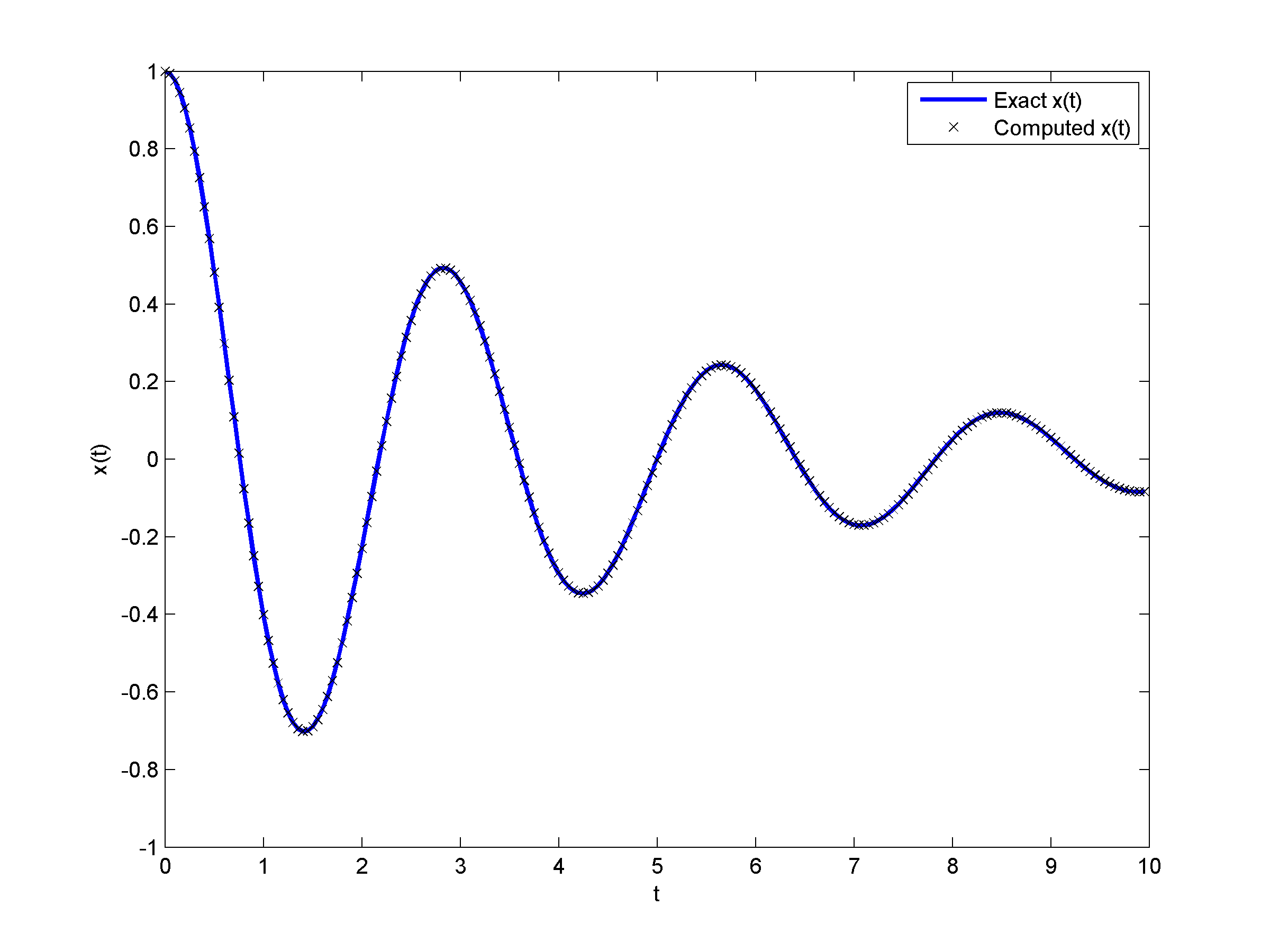}
\caption{\label{NumSolution}Exact vs. computed solution with $N=200$ points.}
\end{figure}

We also compute the error in $x$ at time $t=10$ given by $\left|x_{exact}\left(t=10\right)-x_{N}^{\tau}\right|$,
where $\tau=10/N$ and $N$ is the total number of points
used in the computation. Figure (\ref{ConvAngle}) shows the second
order asymptotic convergence in $x$. 

\begin{figure}[h!]
\centering
\includegraphics[width=16cm]{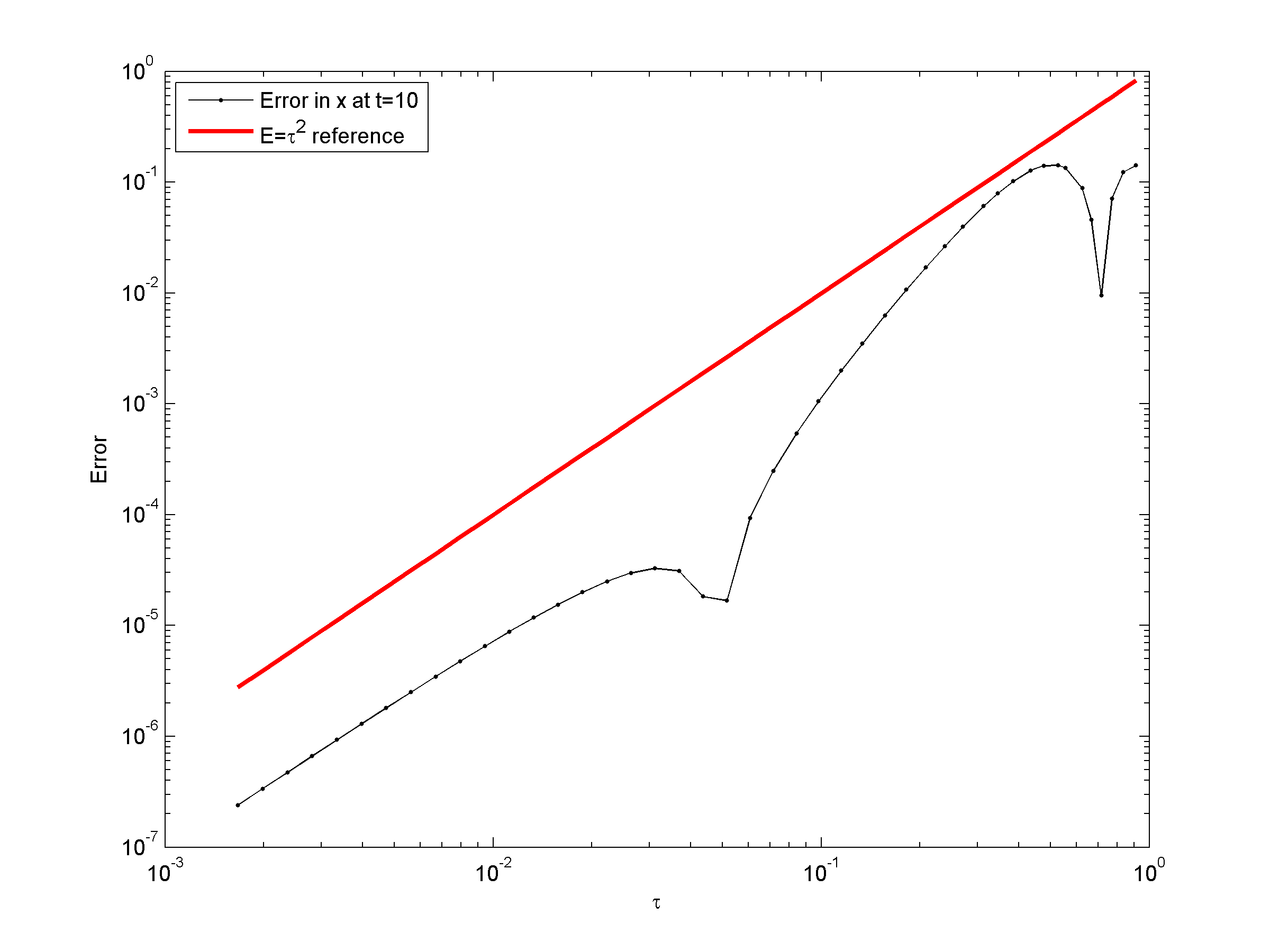}
\caption{\label{ConvAngle}Convergence plot for $x\left(t\right)$.}
\end{figure}

Similarly, we compute the error in $\psi$ at time $t=10$ given by $\left|\psi_{exact}\left(t=10\right)-\psi_{N}^{\tau}\right|$. Figure (\ref{ConvergencePsi}) shows the second order asymptotic convergence in $\psi$.

\begin{figure}[h!]
\centering
\includegraphics[width=16cm]{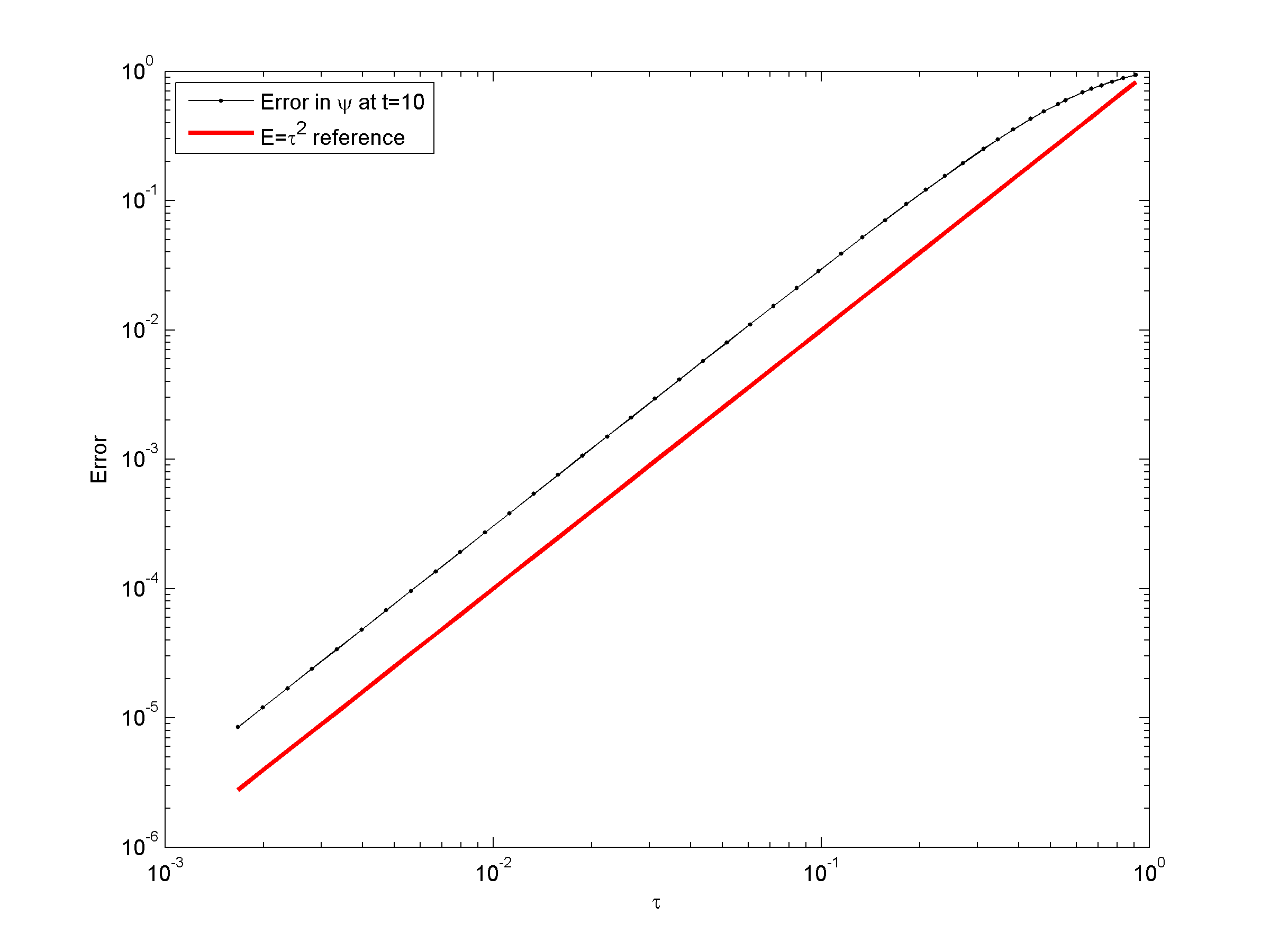}
\caption{\label{ConvergencePsi}Convergence plot for $\psi\left(t\right)$.}
\end{figure}

Finally, for a fixed number of points $N=200$, we observed that $\left|\underset{i=1\cdots N-1}{\max}\psi_{i}^{\tau}-\underset{i=1\cdots N-1}{\min}\psi_{i}^{\tau}\right|\thickapprox5.4\cdot10^{-14},$
thus confirming the discrete conservation property of the multiplier method presented in this paper.
}
\newpage
\rev{
\section{Proofs of some lemmas}

\begin{proof}[\bf Proof of Lemma \ref{lem:taylorExp}:\nopunct]\text{ }\\
Since $\lambda^\tau$ and $D^\tau_t\psi$ are finite difference discretizations of $\lambda$ and $D_t\psi$ with accuracy of order $p$, then for any $w\in C^{q+\max(k,r)}(\overline{\mathcal{I}})$, by Taylor's Theorem, there are remainder terms $C_\lambda, C_t$ depending on $w^{(q+\max(k,r))}$ at $\xi_k,\eta_k$ belonging in a neighbourhood of $t_i$ such that,
\begin{subequations}
\begin{align}
\lambda\contdep{w}{t_i} &= \lambda^\tau\discdep{w}{i}+C_\lambda(w^{(p+1)}(\xi_k))\tau^q, \label{lem:lambdaExp}\\
D_t\psi\contdep{w}{t_i} &= D_t^\tau\psi\discdep{w}{i}+C_t(w^{(p+1)}(\eta_k))\tau^q. \label{lem:DpsiExp}
\end{align}\end{subequations}Setting $w=u+\epsilon v \in C^{q+\max(k,r)}(\overline{\mathcal{I}})$ in \eqref{lem:lambdaExp} and \eqref{lem:DpsiExp} yields \eqref{lem:lambdaRem} and \eqref{lem:DpsiRem}.
\\To show \eqref{lem:lambdaiExp}, Taylor expand $\lambda$ and $\lambda^\tau$ with remainder terms as,\begin{subequations}
\begin{align}
\lambda\contdep{u+\epsilon v}{t_i} &= \lambda\contdep{u}{t_i}+\lambda_1\contdep{u,v}{t_i}\epsilon+\dots + \lambda_{n-1}\contdep{u,v}{t_i}\epsilon^{n-1}+ \lambda_n\contdep{u+\epsilon_1 v}{t_i}\epsilon^n, \label{lem:lambdaEps} \\
\lambda^\tau\discdep{u+\epsilon v}{i} &= \lambda^\tau\discdep{u}{i}+\lambda^\tau_1\discdep{u,v}{i}\epsilon+\dots + \lambda^\tau_{n-1}\discdep{u,v}{i}\epsilon^{n-1}+ \lambda^\tau_n\discdep{u+\epsilon_2 v}{i}\epsilon^n, \label{lem:lambdaDiscEps}
\end{align}\end{subequations}for some $\epsilon_1,\epsilon_2 \in (-|\epsilon|, |\epsilon|)$ and $n=q+\max(k,r)$.
Subtracting \eqref{lem:lambdaEps} with \eqref{lem:lambdaDiscEps} and combining with \eqref{lem:lambdaRem} implies for $1\leq j\leq n-1$,
\begin{align}
C_\lambda(\epsilon) \tau^q = \lambda\contdep{u}{t_i}-\lambda^\tau\discdep{u}{i}+(\lambda_1\contdep{u,v}{t_i}-\lambda^\tau_1\discdep{u,v}{i})\epsilon+\dots+(\lambda_{j}\contdep{u,v}{t_i}-\lambda^\tau_{j}\discdep{u,v}{i})\epsilon^{j}+\mathcal{O}(\epsilon^{j+1}). \label{lem:lambdaDiff}
\end{align} Taking the $j$-th derivative with respect to $\epsilon$ on both sides of \eqref{lem:lambdaDiff} and evaluating it at $\epsilon=0$ yields \eqref{lem:lambdaiExp}. A similar argument shows \eqref{lem:DpsiiExp}.
\end{proof}

\begin{proof}[\bf Proof of Lemma \ref{scalarZClemma}:\nopunct]\text{ }\\
Since $\lambda^{\tau}$ and $D_t^{\tau}\psi$ are zero-compatible of order $l$, then by Taylor expanding about $u_{i+1}=z$, there exists $\xi, \eta$ between $u_{i+1}$ and $z$ such that,
\begin{equation}
\frac{D_t^{\tau}\psi\discdep{u^\tau}{i}}{\lambda^{\tau}\discdep{u^\tau}{i}} = \frac{0+\dots+0+\frac{1}{l!}\left.\frac{\partial^l D_t^{\tau}\psi}{(\partial u_{i+1})^l}\discdep{u^\tau}{i}\right|_{u_{i+1}=\xi} (u_{i+1}-z)^{l}}{0+\dots+0+\frac{1}{l!}\left.\frac{\partial^l \lambda^{\tau}}{(\partial u_{i+1})^l}\discdep{u^\tau}{i}\right|_{u_{i+1}=\eta} (u_{i+1}-z)^{l}} = \frac{\left.\frac{\partial^l D_t^{\tau}\psi}{(\partial u_{i+1})^l}\discdep{u^\tau}{i}\right|_{u_{i+1}=\xi}}{\left.\frac{\partial^l \lambda^{\tau}}{(\partial u_{i+1})^l}\discdep{u^\tau}{i}\right|_{u_{i+1}=\eta}}. \label{scalarZCexpand}
\end{equation} Since $\frac{\partial^l \lambda^{\tau}}{(\partial u_{i+1})^l}\discdep{u^\tau}{i}\neq 0$ is continuous at $u_{i+1}=z$, $\frac{\partial^l \lambda^{\tau}}{(\partial u_{i+1})^l}\discdep{u^\tau}{i}\neq 0$ for sufficiently close $u_{i+1}=\eta$ to $z$. Thus, we obtain the desired result in the limit as $u_{i+1}\rightarrow z$ in \eqref{scalarZCexpand}.
\end{proof}

}
\end{appendices}
}

\end{document}